\documentclass[11pt]{amsart}
\usepackage{preamble}

\usepackage[headheight=15pt, headsep=15pt, footskip=27pt, bottom=2.4cm, left=2.4cm, right=2.4cm]{geometry}

\begin{document}

\title[Invariance of plurigenera for g-pairs with abundant nef parts]{Invariance of plurigenera for generalized polarized pairs with abundant nef parts}


\begin{abstract}
We show the invariance of plurigenera for generalized polarized pairs with abundant nef parts and generalized canonical singularities. This is obtained by investigating a type of newly introduced multiplier ideal sheaf which is of bimeromorphic nature.
\end{abstract}

\author{Zhan Li}
\address[Zhan Li]{Department of Mathematics, Southern University of Science and Technology, 1088 Xueyuan Rd, Shenzhen 518055, China} \email{lizhan@sustech.edu.cn}

\author{Zhiwei Wang}
\address[Zhiwei Wang]{School of Mathematical Sciences, Beijing Normal University, Beijing 100875, China} \email{zhiwei@bnu.edu.cn}

\maketitle

\setcounter{tocdepth}{1}
\tableofcontents

\section{Introduction}

Let $X \to \De$ be a smooth projective family of complex manifolds over a disc. A celebrated result of Siu (see \cite{Siu98, Siu02}) asserts that the plurigenera $h^0(X_t, \Oo_{X_t}(mK_{X_t}))$ is a constant function of $t \in \De$ for each positive integer $m$. Results of such type have been investigated for a long time (see \cite{Kaw99b} for a survey) and have been generalized in many directions after Siu's work (see \cite{Pau07, Tak07, BB12, RT20}, etc.). Invariance of plurigenera is an important ingredient in the proof of the boundedness of moduli of general type varieties (see \cite[Theorem 1.8]{HMX13}, \cite[\S 2.3]{HMX18}). Moreover, Siu's method has a profound impact on problems such as deformations of canonical singularities (see \cite{Kaw99}) and the existence of good minimal models (see \cite{DHP13}).

Let $Y$ be a smooth projective variety and $\mu: Y \dto X$ be the canonical model of $Y$ (see \cite[Definition 3.6.5]{BCHM10}, \cite[Definition 2.1]{Li20b}). It is natural to ask whether $X$ (birationally) belongs to a bounded moduli space given $Y$ satisfying certain restrictions (see \cite{FS20}, \cite[Conjecture 1.2]{Li20a}). When $Y$ is of general type, then the answer is affirmative if $Y$ has a fixed dimension and a bounded volume (see \cite{HMX13, HMX18}). To study the general situation, one can adopt the same strategy as \cite{HMX13, HMX18}. In this case, the base variety $X$ admits a generalized polarized pair structure instead of a log pair structure (see Definition \ref{def: gpair}). Roughly speaking, there is a triple $(X, B+M)$ with $B, M$ divisors on $X$. Here $B \geq 0$ accounts for the singularities of the fibers, and $M$ is the push-forward of a nef and abundant divisor which accounts for the moduli of the fibers. It is this additional $M$ that ruins the straightforward applications of the results in \cite{HMX13, HMX18}. For one thing, $M$ is only well-defined up to linear equivalence.

More generally, if $Y \to X$ is a $K_Y$-trivial fibration, then $X$ also admits a generalized polarized pair structure. In fact, the concept of generalized polarized pairs originates from such observations (see \cite{Bir20}). Note that in the definition of generalized polarized pairs, the nef part is just assumed to be the push-forward of a nef divisor instead of a nef and abundant divisor. However, the latter case is more meaningful in geometry, and our main motivation is to study the invariance of plurigenera in such setting:

\begin{theorem}\label{thm: sing extension}
Let $\pi: X \to \De$ be a projective contraction from a  complex space $X$ to the disc $\De$. Assume that $X$ has canonical singularities. Let $L$ be a Cartier divisor on $X$ and $h$ be a metric for $\Oo_X(L)$ with non-negative curvature current. Let $X_0\subset X$ be the fiber over $0\in\De$. Suppose that $X_0$ has canonical singularities and $h|_{X_0}$ is well-defined. Then for each $m\in \Nn$, any section of
\[
H^0(X_0, \Oo_{X_0}(mK_{X_0}+L|_{X_0}) \otimes \Gg_m(h|_{X_0}))
\] extends over $X$ (i.e. any section has a preimage in $H^0(X, \Oo_X(mK+L))$ under the map in Lemma \ref{le: meaning of extension}).
\end{theorem}

The newly introduced ideal sheaf $\Gg_m(h)$ is of the birational (or bimeromorphic) nature (see Section  \ref{subsec: Multiplier ideal sheaves and sections}). Comparing with previous studies, this birational point of view gives new ingredients even in the smooth case (see Example \ref{eg: not the same as multiplier ideal sheaf}, Remark \ref{rmk: Paun's original thm is not enough}). For relevant notions of metrics on complex spaces which are adapted to this perspective, see Section \ref{sec: 3}.

Theorem \ref{thm: sing extension} involves with metrics and $L^2$-conditions. In order to use it in algebraic geometry, we need to obtain such analytic requirements by natural algebro-geometric conditions. The following statement is metric-free.

\begin{theorem}\label{thm: AG sing trivial boundary extension}
Let $\pi: X \to \De$ be a projective contraction from a  complex space $X$ to the disc $\De$. Let $X' \xrightarrow{f} X \xrightarrow{\pi} \De$ be a generalized polarized pair with the boundary part $B$ and the abundant nef part $M$. Let $X_0 \subset X$ be the fiber over $0\in\De$ and $(X_0, B_0+M_0)$ be the generalized polarized pair obtained by restricting to $X_0$. Assume that $(X_0, B_0+M_0)$ has g-canonical singularities (in particular, $B_0=0$). Then for each $m\in\Nn$ such that $mM$ is Cartier, any section of
\[
H^0(X_0, \Oo_{X_0}(m(K_{X_0}+ M|_{X_0})))
\] extends over $X$.
\end{theorem}

A direct consequence of Theorem \ref{thm: AG sing trivial boundary extension} is

\begin{corollary}\label{cor: sing invariant of plurigenera for g-pair}
Let $\pi: X \to \De$ be a projective contraction from a complex space $X$ to the disc $\De$. Let $X' \xrightarrow{f} X \xrightarrow{\pi} \De$ be a generalized polarized pair with the boundary part $B$ and the abundant nef part $M$. Let $X_t\subset X$ be the fiber over $t\in\De$ and $(X_t, B_t+M_t)$ be the generalized polarized pair obtained by restricting to $X_t$. Assume that for any $t\in \De$, $(X_t, B_t+M_t)$ has g-canonical singularities (in particular, $B_t=0$). Then for each $m\in\Nn$ such that $mM$ is Cartier, 
\[
h^0(X_t, \Oo_{X_t}(m(K_{X_t}+M|_{X_t})))
\] is independent of $t \in \De$.
\end{corollary}

Both the abundant assumption and the singularity assumption are indispensable. In fact, Corollary \ref{cor: sing invariant of plurigenera for g-pair} is false if $M$ is just assumed to be nef (see Example \ref{eg: M nef}). An example of Kawamata shows that the extension of local sections fails if the singularities are klt (see \cite[Example 4.3]{Kaw99b}).

It turns out that even in the smooth case, P\u aun's twisted version of invariance of plurigenera (see \cite[Theorem 1]{Pau07}) is not enough to obtain Theorem \ref{thm: AG sing trivial boundary extension} (see Remark \ref{rmk: Paun's original thm is not enough}). Instead, we need Theorem \ref{thm: sing extension} which allows potentially more sections to be extended.

Besides, \cite[Theorem 4.1]{FS20} establishes a version of invariance of plurigenera assuming that $K_X+B+M$ or $B+M$ is big over $\De$ and $M$ is nef over $\De$. The argument of \cite[Theorem 4.1]{FS20} follows from \cite[Theorem 4.2]{HMX13} which relies on the minimal model program for varieties of general type.

We discuss the structure of the paper. Section \ref{sec: 2} gives the background material on generalized polarized pairs for complex spaces. Section \ref{sec: 3} discusses metrics of $\Qq$-Cartier divisors on complex spaces. Section \ref{sec: 4} gives the construction of metrics for the abundant nef parts of generalized polarized pairs. Section \ref{sec: 5} introduces a new type multiplier ideal sheaf and proves the aforementioned theorems and corollary. 

\medskip

\noindent\textbf{Acknowledgements}.
We thank Hanlong Fang, Sheng Rao and Lei Zhang for many discussions and for answering questions. Z. L. is partially supported by a grant from SUSTech. Z. W. is partially supported by  National Key R\&D Program of China (No.2021YFA1002600) and  NSFC grant (No.12071035).

\section{Generalized polarized pairs}\label{sec: 2}

\subsection{Notation and conventions}
Let $\Zz$ be the set of integers and $\Nn= \Zz_{>0}$ be the set of natural numbers.

A holomorphic map $f: X \to Y$ between complex spaces is called a morphism, and a morphism is called a contraction if it is surjective with connected fibers. The morphism $f$ is called a proper modification if (1) $f$ is proper and surjective, and (2) there exists a nowhere dense analytic subset $Z\subset Y$ such that $f|_{X- f^{-1}(Z)}: X- f^{-1}(Z) \to Y- Z$ is an isomorphism. A bimeromorphic map $f: X\dto Y$ between complex spaces is a meromorphic map such that the graph $\Gamma_f$ is an irreducible analytic subset in $X \times Y$, and the natural maps $\Gamma_f \to X, \Gamma_f \to Y$ are proper modifications (see \cite[Definitions 2.1, 2.2, 2.7]{Uen75}). A proper modification $\mu: W \to X$ is called a resolution if $W$ is smooth. A proper morphism $f: X \to Y$ between complex spaces is called projective if for any relatively compact open set $U \subset Y$, there is an embedding $j: f^{-1}(U) \hookrightarrow \Pp_U^n \coloneqq \Pp_{\Cc}^n \times U$ such that $f|_{f^{-1}(U)}=p_2\circ j$, where $p_2: \Pp_U^n \to U$ is the natural projection map (see \cite[Chapter V]{Pet94}).

A divisor means a Weil divisor. A finite $\Qq$-linear combination of divisors gives $\Qq$-divisors. A $\Qq$-divisor is called $\Qq$-Cartier if it is a finite $\Qq$-linear combination of Cartier divisors.

Let $D$ be a $\Qq$-divisor on $X$. Then $(X, D)$ is called a log pair. A resolution $\mu: W \to X$ is called a log resolution of $(X, D)$ if $W$ is smooth and $\Supp \mu_*^{-1} D \cup \Supp \Exc(\mu)$ is a simple normal crossing divisor, where $\mu_*^{-1} D$ is the strict transform of $D$ and $\Exc(\mu)$ is the exceptional locus of $\mu$. When $X$ is a reduced complex space, $(X, D)$ always admits a projective log resolution (see \cite[Theorem 2.0.2]{Wlo09}).

Let $X$ be a normal complex space. Let $X \to S$ be a morphism to a complex space $S$. We write the relative property over $S$ by $/S$. Let $k = \Zz, \Qq$. For two $k$-divisors $B$ and $D$, we use $B \sim_{k } D/S$ to denote that $B$ and $D$ are $k$-linearly equivalent over $S$.

Let $D$ be a Weil divisor on a normal complex space $X$, then $\Oo_X(D)$ denotes the sheaf associated with $D$. To be precise, let $\mathscr{M}_X$ be the sheaf of germs of meromorphic functions on $X$ (see \cite[Chapter II, \S 6.2]{Dem12}), then $\Oo_X(D)\subset \mathscr{M}_X$ is the sheaf such that for any open set $U \subset X$,
\[
\Oo_X(D)(U) = \{f\in \mathscr{M}_X(U) \mid  \di(f)+D|_{U} \geq 0\}.
\] The sheaf $\Oo_X(D)$ is coherent and the subscript of $\Oo_X(D)$ will be omitted if it is clear from the context.

Recall that for an $n$-dimensional normal complex space $X$, the canonical divisor $K_X$ is any Weil divisor such that $j_*\Omega^n_{X_{\rm reg}} \simeq \Oo_X(K_X)$ for the open embedding $j: X_{\rm reg} \hookrightarrow X$, where $X_{\rm reg}$ is the smooth locus of $X$ and $\Omega^n_{X_{\rm reg}}$ is the sheaf of holomorphic $n$-forms on $X_{\rm reg}$.

Let $f: X \to Y$ be a morphism and $B$ be a $\Qq$-Cartier divisor on $Y$. We write $f^*B$ for the $\Qq$-Cartier divisor which is the pull-back of $B$. If $D$ is a Cartier divisor on $X$, then we write $\Oo_X(D)$ for the corresponding line bundle, and $f_*\Oo_X(D)$ for the push-forward of the sheaf $\Oo_X(D)$. 

Let $f: X \to Y$ be a proper modification between normal complex spaces and $D$ be a prime divisor on $X$. We define the push-forward divisor 
\[
f_*D \coloneqq\begin{cases}
f(D), & \text{ if~ } D \text{~is not~} f\text{-exceptional},\\
0, & \text{ if~ } D \text{~is~} f\text{-exceptional}.
\end{cases}
\] This construction can be extended to $\Qq$-divisors by linearity. Moreover,  for $k = \Zz, \Qq$, if $B \sim_{k } D/S$, then $f_*B \sim_{k } f_*D/S$ because the push-forward of a principal divisor is still a principal divisor. 

Finally, for a projective morphism $X \to S$, a $\Qq$-Cartier divisor $D$ on $X$ is nef$/S$ if there is an ample$/S$ divisor $H$ on $X$ such that $D+\ep H$ is ample$/S$ for any $\ep \in \Qq_{>0}$, and $D$ is big$/S$ if $D|_{X_t}$ is big for a general $t\in S$.
 
\subsection{Generalized polarized pairs with abundant nef part}

Generalized polarized pairs originate from the canonical bundle formula (see \cite{Kaw98}). It was observed by \cite{BZ16} that the generalized polarized pair structure deserves to be studied separately and it behaves like the usual log pair structure in many ways. We state the definition of generalized polarized pairs for complex spaces (see {\cite[Definition 1.4]{BZ16} for the original definition in the algebraic category).

\begin{definition}\label{def: gpair}
Let $X', X$ and $S$ be normal complex spaces. A generalized polarized pair (g-pair) consists of  projective morphisms $X' \xrightarrow{f} X \to S$ where $f$ is a modification, a $\Qq$-divisor $B \geq 0$ on $X$, and a $\Qq$-Cartier divisor $M'$ on $X'$ which is nef$/S$ such that $K_{X} + {B} + {M}$ is $\Qq$-Cartier, where $M \coloneqq f_*M'$. We call $B$ the boundary part and $M$ the nef part. 
\end{definition}

\begin{remark}\label{rmk: replacing by higher model}
In the definition of g-pairs, we can replace $X'$ by any projective modification of $X'$ and $M'$ by its pull-back. Therefore, we can assume that $f$ is a projective log resolution of $(X, B)$. 
\end{remark}

\begin{definition}\label{def: abundant gpair}
Under the notation of Definition \ref{def: gpair}, if there exists a proper contraction $g: X' \to Z/S$ such that  $M' \sim_\Qq g^*H/S$ for some nef and big$/S$ $\Qq$-Cartier divisor $H$ on $Z$, then we call that such g-pair has the abundant nef part.
\end{definition}

Replacing $X'$ by a projective modification and $M'$ by its pull-back (see Remark \ref{rmk: replacing by higher model}), the new g-pair still has the abundant nef part.

\begin{remark}
A g-pair with the abundant nef part naturally appears in the canonical bundle formula where the concept of g-pairs originates (see \cite{Amb05}). The word ``abundant" comes from the fact that divisors satisfying the property in Definition \ref{def: abundant gpair} is related to the abundance conjecture (see \cite[Proposition 2.1]{Kaw85}).
\end{remark}

\subsection{Singularities and adjunctions for g-pairs}

\begin{definition}\label{def: Q-Gorenstein}
A normal complex space is called $\Qq$-Gorenstein if $K_X$ is a $\Qq$-Cartier divisor. 
\end{definition}

\begin{definition}\label{def: canonical sing}
A normal complex space $X$ has canonical singularities if (1) $X$ is $\Qq$-Gorenstein, and (2) for any resolution $f: W \to X$, if $K_W=f^*K_X+E$ with $E$ an $f$-exceptional divisor, then $E \geq 0$.
\end{definition}

\begin{definition}\label{def: discrepancies}
Under the notation of Definition \ref{def: gpair}, for a prime divisor $P$ over $X$, its discrepancy with respect to (X, B+M) is defined to be
\[
{\rm discp}(P;X,B+M) \coloneqq \mult_P\left(K_Y+g^*M' - (f\circ g)^*(K_X+B+M)\right),
\] where $g: Y \to X'$ is a proper modification such that $P \subset Y$.
\end{definition}

\begin{definition}\label{def: singularities}
A g-pair $(X, B+M)$ has g-canonical (resp. g-terminal, g-klt, g-lc) singularities if ${\rm discp}(P;X,B+M) \geq 0$ (resp. $>0$, $>-1$, $\geq -1$) for each prime divisor $P$ over $X$.
\end{definition}

The adjunction formula for g-pairs is given by \cite[Definition 4.7]{BZ16} in the algebraic setting. It can be naturally extended to the analytic setting. 

\begin{definition}[Adjunction formula for g-pairs]\label{def: g-adjunction}
	Let $(X,B+M)$ be a g-pair with data $ X' \xrightarrow{f} X \to S$ and $M'$. Let $P$ be a normal irreducible component of $\lfloor B \rfloor$ and $P'$ be its strict transform on $X'$. We may assume that $f$ is a projective log resolution of $(X,B+M)$. Write 
	\[
	K_{X'} +B'+M'=f^*(K_{X} +B +M),
	\] then
	\[
	K_{P'} +B_{P'} +M_{P'} \coloneqq (K_{X'} +B'+M')|_{P'},
	\]
	where $B_{P'} = (B'-P')|_{P'}$ and $M_{P'} =M'|_{P'}$. Let $g$ be the induced morphism $P'\to P$. Set $B_{P} = g_*B_{P'}$ and $M_{P} =g_*M_{P'}$. Then we get the equality
	\[
	K_{P}+B_{P}+M_{P} = (K_{X}+B+M)|_{P},
	\] which is referred as (generalized) adjunction formula.
\end{definition}

In general, $M|_{P} \neq M_P$, and even if $(X,B+M)$ has the abundant nef part, $(P, B_P+M_P)$ may not have the abundant nef part. On the other hand, in the category of algebraic varieties, if $(X, B+M)$ is g-lc, then $(P, B_P+M_P)$ is still g-lc with data $P' \xrightarrow{f} P \to S$ and $M_{P'}$ by \cite[Remark 4.8]{BZ16}.

Let $X \to \De$ be a projective contraction to the disc $\De$. Let $X_0$ be the fiber over $0\in\De$. Suppose that $X_0$ is normal, then we define
\begin{equation}\label{eq: adjunction on X_0}
(X_0, B_0+M_0)
\end{equation} to be the g-pair by the adjunction of $(X, X_0+B+M)$ on $X_0$. In the sequel, $(X_0, B_0+M_0)$ will be assumed to have g-canonical singularities. Under this assumption, $B_0 =0$ and $M_0 = M|_{X_0}$.

\section{Metrics on complex spaces}\label{sec: 3}

\subsection{Metrics and $\Qq$-metrics}
Let $X$ be a normal complex space, $\Ll$ be a holomorphic line bundle on $X$. We define the notion of singular metrics in this setting.

Let $x\in X$ be a point and $\Ll_x$ be the stalk at $x$. Assume that there is a function
\[
\|-\|_h: \Ll_x \to \Rr_{\geq 0} \cup \{+\infty\},
\] such that if $\theta: \Ll|_U \simeq \Oo_U$ is a trivialization on an open set $U \subset X$, then for any $s \in \Ll_x$, we have
\[
\|s\|_h = |\theta(s)|e^{-\vphi}
\] for a function $\vphi: U \to \Rr \cup \{\pm \infty\}$. This $\vphi$ is called a local weight and it depends on $\theta$. Let $\mu: W \to X$ be a resolution. Then $\mu^*\theta: \mu^*\Ll|_{\mu^{-1}(U)} \simeq \Oo_{\mu^{-1}(U)}$ is a trivialization for $\mu^*\Ll$. For any $\ti s \in (\mu^*\Ll)_w$ with $w \in \mu^{-1}(U)$, set
\begin{equation}\label{eq: pullback of metric}
\|\ti s\|_{\mu^*h} \coloneqq |\mu^*\theta(\ti s)|e^{-\mu^*\vphi},
\end{equation} where $\mu^*\vphi(w) = \vphi(\mu(w))$. Note that $\|-\|_{\mu^*h}$ is independent of the choice of trivializations. In fact, for local trivializations $\theta_1, \theta_2$, there is a nowhere vanishing holomorphic function $u$ such that $\theta_1=u\theta_2$. Thus $|u|e^{-\vphi_1} = e^{-\vphi_2}$ and 
\begin{equation}\label{eq: pullback relation}
|\mu^*u|e^{-\mu^*\vphi_1} = e^{-\mu^*\vphi_2}.
\end{equation}

Let $L^1_{\rm loc}(V)$ be the set of locally integrable functions on a smooth open set $V$.

\begin{definition}[Singular metrics on normal complex spaces]\label{def: singular metric on varieties}
Under the above notation, $h$ is a (singular) metric of $\Ll$ if for any resolution $\mu: W \to X$, we have $\mu^*\vphi \in L_{\rm loc}^1(\mu^{-1}(U))$ for any open set $U \subset X$.
\end{definition}

The property of $\vphi$ in Definition \ref{def: singular metric on varieties} will be referred as the $L^1_{\rm loc}$-property. By abuse of terminology, if $D$ is a Cartier divisor, then we also call a metric of the corresponding line bundle $\Oo(D)$ a metric of $D$. This is particularly convenient to treat $\Qq$-Cartier divisors.

\begin{definition}[{\cite[Definition 1.4]{Pet94}}]\label{def: psh on complex space}
Let $X$ be a complex space. A plurisubharmonic (psh for short) function on $X$ is a function $\varphi:X\rightarrow [-\infty,\infty)$ having the following property. For every $x\in X$ there is an open neighborhood $U$ with a biholomorphic map $h:U\rightarrow V$ onto a closed complex subspace $V$ of some domain $G\subset \mathbb C^m$ and a  plurisubharmonic function $\tilde \varphi:G\rightarrow [-\infty,\infty)$ such that $\varphi|_U=\tilde\varphi\circ h$.
\end{definition}

If $f:X\rightarrow Y$ is a holomorphic map between two complex spaces, and $\varphi:Y\rightarrow [-\infty,\infty)$ is a psh function on $Y$, then $\varphi\circ f$ is a psh function on $X$ (see \cite[Page 356]{Nar61}).

 \begin{definition}\label{def: non-negative curvature on variety}
Under the notation of Definition \ref{def: singular metric on varieties}, $(\Ll, h)$ is said to have non-negative curvature current if any local weight $\varphi$  on an open subset of $X$ is a psh function.
 \end{definition}

 \begin{remark}\label{rmk: change of variable}
 	In Definition \ref{def: singular metric on varieties}, the local $L_{\rm loc}^1$-property is needed to hold for all resolutions. If the $L_{\rm loc}^1$-property holds for one resolution $\mu: W \to X$, then it is not true that this property holds for another resolution $\nu: V \to X$ even if $\nu$ factors through $\mu$. In fact, by the change-of-variables formula, 
 	\[
 	\int_W |\ti\vphi| ~dV_W = \int_V |\nu^*\ti\vphi| \cdot |s|^2~dV_V,
 	\] where $s$ is the local equation for the effective exceptional divisor $K_V-f^*K_W$ (hence $s$ is holomorphic). Hence $\ti\vphi \in L_{\rm loc}^1$ does not imply $\nu^*\ti\vphi \in L_{\rm loc}^1$.
 	
 	However, if the curvature current is non-negative (see Definition \ref{def: non-negative curvature on variety}), then it is enough to just consider one resolution (see Proposition \ref{prop: independent}).
 \end{remark}

The notion of psh functions is of bimeromorphic nature:
 
 \begin{proposition}\label{prop: independent}
 Let $\mu: W \to X$ be a resolution of a normal complex space $X$. Then a function $\vphi: X \to [-\infty, \infty]$ is a psh function iff $\vphi\circ\mu: Y \to [-\infty, \infty]$ is a psh function.
\end{proposition}
\begin{proof}
One only needs to show the sufficient part. As $\mu$ is a proper modification, $Z \coloneqq \mu(\Exc(\mu))$ is an analytic set of codimension $\geq 2$. Then $\vphi$ is a psh function on $X-Z$. By \cite[Theorem 1.5]{Pet94}, the definition of psh functions in  \cite{GR56} is the same as Definition \ref{def: psh on complex space}. By \cite[Page 181, Satz 4]{GR56}, $\vphi$ extends to a psh function $\hat\vphi$ on $X$. As $\hat\vphi\circ\mu$ is a psh function on $W$ which coincides with $\vphi\circ\mu$ on $W-\Exc(\mu)$, we have $\hat\vphi\circ\mu=\vphi\circ\mu$ on $W$. Hence $\vphi=\hat\vphi$ is a psh function.
\end{proof}

For simplicity, we can formally extend the notion of metrics for $\Qq$-Cartier divisors (or $\Qq$-line bundles). 

\begin{definition}\label{def: metric for Q-div}
Let $D$ be a $\Qq$-Cartier divisor on a complex space $X$. Then a $\Qq$-metric $h$ for $D$ is a triple $(m, D, h_{mD})$ such that $m \in \Nn$ with $mD$ a Cartier divisor and $h_{mD}$ is a metric for $mD$. We say that two $\Qq$-metrics $(m, D, h_{mD})$ and $(n, D, h_{nD})$ are equal if $h_{mD}^n= h_{nD}^m$ as metrics for the Cartier divisor $mnD$. For simplicity, $h$ is also called a metric.
\end{definition}

\begin{remark}\label{rmk: uniqueness}
Given a metric $(m, D, h_{mD})$, for any $n \in \Nn$ such that $nD$ is Cartier, there always exists a metric $(n, D, h_{nD})$ which is equal to $(m, D, h_{mD})$. In fact, we can set $h_{nD} = h_{mD}^{\frac n m}$.
\end{remark}

If $D =\sum D_i$ is a finite sum of $\Qq$-Cartier divisors such that each $D_i$ admits a metric $h_i$, then we define a metric $h\coloneqq \prod_i h_{i}$ for $D$ as follows. Take $m\in\Nn$ such that each $mD_i$ is Cartier, then 
\[
h=(m, D, h_{mD})
\] with $h_{mD} = \prod_i h_{mD_i}$ where $h_i=(m, D_i, h_{mD_i})$ (see Remark \ref{rmk: uniqueness}). It is straightforward to check that this definition is independent of the choice of $m$. Let $D$ be a $\Qq$-Cartier divisor with a metric $h$, then for $r\in \Qq$, we can similarly define a metric $h^r$ for $rD$. 

Let $f: W \to X$ be a proper modification between normal complex spaces. Let $D$ be a $\Qq$-Cartier divisor on $X$ with a metric $h$. We can define the pull-back metric $f^*h$ for the $\Qq$-Cartier divisor $f^*D$. In fact, it is enough to assume that $D$ is a Cartier divisor by the definition of metrics for $\Qq$-Cartier divisors. If $\theta: \Oo(D)|_U \simeq \Oo_U$ is a trivialization, then $f^*\theta: \Oo(f^*D)|_{f^{-1}(U)} \simeq \Oo_{f^{-1}(U)}$. If $\|s\|_{h}=|\theta(s)|e^{-\vphi}$ for $x\in U$ and $s\in \Oo(D)_x$, then
\[
\|\ti s\|_{f^*h} \coloneqq |f^*\theta(\ti s)| e^{-f^*\vphi}
\] for $w\in f^{-1}(U)$ and $\ti s\in \Oo(f^*D)_w$. This definition is independent of the choice of $\theta$.

\begin{remark}\label{rmk: in terms of weight}
It is more convenient to think of a metric in terms of a local weight: if $\vphi$ is the local weight for $h_{mD}$ (under a given trivialization), then $\frac \vphi m$ is defined to be the local weight of $(m, D, h_{mD})$. The summation  (resp. rational multiple,  pull-back) of local weights corresponds to the product (resp. rational exponent, pull-back) of metrics. 
\end{remark}

 \begin{remark}\label{rmk: metric for linear equiv}
Suppose that $B, D$ are $\Qq$-Cartier divisors such that $B \sim_\Qq D$, then a metric of $B$ is also a metric of $D$. In fact, local weights are identified under an isomorphism $\Oo(mB) \simeq \Oo(mD)$ for some $m\in \Nn$.
\end{remark}

On a complex manifold, a function $\vphi$ is called quasi-plurisubharmonic (quasi-psh) if it is a summation of a psh function with a smooth function.

\begin{definition}[{\cite[Definition 1.4]{DPS00}}]\label{def: compare sing}
For two quasi-psh functions $\vphi_1, \vphi_2$, we say that $\vphi_1$ is less singular than $\vphi_2$ (and write $\vphi_1 \preceq \vphi_2$) if $\vphi_2 \leq \vphi_1+c$ for a constant $c$. We write $\vphi_1 \approx \vphi_2$ if $\vphi_1 \preceq \vphi_2$ and $\vphi_2 \preceq \vphi_1$.

For two functions $\vphi_i: X \to \Rr \cup \{-\infty\}, i=1,2$ on a complex space $X$, we use the same notation as above if the corresponding relations hold after pulling back $\vphi_i$ to all the resolutions.
\end{definition}

\subsection{Metrics defined by global sections}\label{subsection: metrics defined by global sections}

Let $X$ be a normal complex space and $L$ be a Cartier divisor on $X$. Suppose that $\sigma_i \in H^0(X, \Oo(L)), 1 \leq i \leq k$ are global sections (may not necessarily be distinct). They can be used to define a  metric $h$ for $L$ just as in the smooth case.

Let $U \subset X$ be an open set such that $\theta: \Oo(L)|_U \simeq \Oo_U$ is a trivialization. Then for any $s\in L_x$, set
\begin{equation}\label{eq: definition of metric}
\|s\|^2_h = \frac{|\theta(s)|^2}{\sum_{1 \leq i \leq k} |\theta(\sigma_i)|^2}=|\theta(s)|^2 e^{-2\vphi}.
\end{equation} The local weight of $h$ (under the trivialization $\theta$) is
\[
\vphi = \frac 1 2 \log (\sum_{1 \leq i \leq k} |\theta(\sigma_i)|^2).
\] It satisfies the $L_{\rm loc}^1$-property for any resolution $\mu: W \to X$ because 
\[
\mu^*\vphi =\frac 1 2 \log (\sum_{1 \leq i \leq k} |\mu^*\theta(\mu^*\sigma_i)|^2)
\] is a psh function and thus Proposition \ref{prop: independent} applies to this situation.

More generally, for a $\Qq$-Cartier divisor $D$ on  $X$, suppose that $mD$ is Cartier and $\sigma_i \in H^0(X, \Oo_{X}(mD)), 1\leq i\leq k$ are global sections. If $\theta: \Oo(mD)|_U \simeq \Oo_U$ is a trivialization, then set
\[
\vphi = \frac{1}{2m}\log (\sum_{1\leq i\leq k} |\theta(\sigma_i)|^{2})
\] as the local weight of $h$ (see Remark \ref{rmk: in terms of weight}).

We call the metrics defined above the metrics defined by global sections. 

For an effective Cartier divisor $D$, choose $\sigma=1 \in H^0(X, \Oo_{X}(D))$, then the metric as \eqref{eq: definition of metric} is denoted by $\hbar_D$.

In general, let $D$ be a $\Qq$-Cartier divisor. If $D = \sum_{1 \leq i \leq n} r_i D_i-\sum_{1 \leq j\leq m}s_j B_j$ is a decomposition such that $r_i,s_j \in \Qq_{>0}$ and $D_i, B_j$ are effective Cartier divisors, then set
\begin{equation}\label{eq: for non-effective div}
\hbar_D \coloneqq  \left( \prod_{1 \leq i \leq n} \hbar_{D_i}^{r_i}\right)\cdot  \left( \prod_{1 \leq j \leq m} (\hbar^{-1}_{B_j})^{s_j}\right),
\end{equation} 
where $\hbar^{-1}_{B_j}$ is the dual metric for $\Oo(-B_j)$. Note that $\hbar_D$ is a metric as the corresponding weight still satisfies the $L^1_{\rm loc}$-property. 

The above $\hbar_D$ is independent of the choice of decompositions of $D$. Indeed, suppose that $mD$ is Cartier for some $m \in \Nn$ and $\theta: \Oo(mD)|_U \simeq \Oo_U$ is a trivialization. Then for any $x\in U$ and $s\in \Oo(mD)_x$, we always have $\|s\|^2_{\hbar_{D}^m}=|\theta(s)|^2/|\theta(1)|^2$.

Let $Z \subset Y$ be a normal complex subspace such that $Z \not\subset \Supp D$, then $D|_Z$ is still a $\Qq$-Cartier divisor. Let $\hbar_D|_Z$ be the pull-back of $\hbar_D$ to $Z$, then
\begin{equation}\label{eq:restriction}
\hbar_D|_Z = \hbar_{D|_Z}.
\end{equation}

\begin{proposition}\label{prop: product of metric defined by global sections}
Suppose that $h_B, h_D$ are metrics defined by global sections for $\Qq$-Cartier divisors $B, D$ respectively. Then $h_Bh_D$ is a metric for $B+D$ which is also defined by global sections. 
\end{proposition}
\begin{proof}
Suppose that $h_B$ (resp. $h_D$) is defined by global sections of $\Oo(rB)$ (resp. $\Oo(sD)$), and $\theta_{rB}: \Oo(rB)|_U \simeq \Oo_U$ (resp. $\theta_{sD}: \Oo(sD)|_U \simeq \Oo_U$) is a trivialization. Then the local weights are
\[
\vphi_B = \frac{1}{2r}(\log \sum_{i=1}^{m_B}|\theta_{rB}(\sigma_i)|^2), \quad \vphi_D = \frac{1}{2s}(\log \sum_{j=1}^{m_D}|\theta_{sD}(\tau_j)|^2),
\] where $\sigma_i \in H^0(U, \Oo(rB)), \tau_j \in H^0(U, \Oo(sD))$. The local weight for $B+D$ is
\[
\begin{split}
\vphi_B+\vphi_D &= \frac{1}{2rs}\log\left((\sum_{i=1}^{m_B}|\theta_{rB}(\sigma_i)|^2)^s \cdot (\sum_{j=1}^{m_D}|\theta_{sD}(\tau_j)|^2)^r\right)\\
&=\frac{1}{2rs}\log\left( \sum_{\substack{i_1, \cdots, i_s\\ j_1 \cdots j_r}} |\theta_{rB}(\sigma_{i_1}) \cdots \theta_{rB}(\sigma_{i_s})\cdot  \theta_{sD}(\tau_{j_1}) \cdots \theta_{sD}(\tau_{j_r})|^2\right).
\end{split}
\] By $\sigma_{i_1} \cdots \sigma_{i_s}\cdot  \tau_{j_1} \cdots \tau_{j_r} \in H^0(U, \Oo(rs(B+D)))$, the claim follows.
\end{proof}

\subsection{Push-forward metrics defined by global sections}\label{subsection: pushforward metrics}

Let $f: X' \to X$ be a proper modification between normal complex spaces and $D$ be a $\Qq$-Cartier divisor on $X'$. Suppose that $B = f_*D$ is still a $\Qq$-Cartier divisor on $X'$. Suppose that $D$ admits a metric $h'$. In general, there is no push-forward metric $f_*h'$ of $h'$. However, if $h'$ is given by global sections, then we can naturally define $f_*h'$ as follows.

Recall that for a $\sigma \in \mathscr{M}_{X'}(X')$, $f_*\sigma$ is defined as the meromorphic extension of $\sigma|_{X-f(\Exc(f))}$ on $X$ (see \cite[Chapter II (10.2)]{Dem12}). This is possible as $\codim_Xf(\Exc(f)) \geq 2$. Hence $f_*\sigma \in \mathscr M_X(X)$.

Suppose that $mD$ and $mB$ are Cartier divisors. If $\sigma \in H^0(X', \Oo_{X'}(mD))$, that is, $\sigma\in \mathscr{M}_{X'}(X')$ such that $\di(\sigma) + mD \geq 0$, then $f_*\sigma \in H^0(X, \Oo_X(mB))$ because
 \[
 \di(f_*\sigma)+mB = f_*\di(\sigma)+f_*(mD) \geq 0.
 \] Therefore, if $h'$ is the metric of $D$ defined by global sections as \eqref{eq: definition of metric}, then the push-forwards of these sections also define a metric $h\coloneqq f_*h'$ for $B$. We call this metric the push-forward of $h'$.
 
 \begin{lemma}\label{le: comparing pullback metrics}
Under the above notation and assumptions, let $D+F= f^*B$ with $F$ an $f$-exceptional divisor ($F$ may not be effective). Then $f^*h$ is a metric for $D+F$ such that
\begin{equation}\label{eq: comparing metrics}
h'\cdot \hbar_F = f^*h,
\end{equation}
where $\hbar_F$ is the metric defined as \eqref{eq: for non-effective div}.
\end{lemma}

\begin{proof}
Suppose that $h'$ is defined by sections $\sigma_i \in H^0(X', \Oo_{X'}(mD)), i=1 ,\cdots, k$, then $\sigma_i =f^*(f_*\sigma_i) \in H^0(X', \Oo_{X'}(mD+mF)), i=1 ,\cdots, k$.

Let $\theta: \Oo_{X'}(mD)|_{U} \simeq \Oo_U, \theta_F: \Oo_{X'}(mF)|_U \simeq \Oo_U$ be trivializations on $U\subset X'$, and $\theta_B: \Oo_X(mB)|_V \simeq \Oo_V$ be a trivialization on $V\subset X$ such that $U\subset f^{-1}(V)$ and $f^*\theta_B = \theta \theta_F$. Then $h'\cdot \hbar_F$ has the local weight
\[
\frac{1}{2m} (\log \sum_i |\theta(\sigma_i)|^{2} + \log |\theta_F(1)|^{2}) = \frac{1}{2m} \log (\sum_i |\theta\theta_F(\sigma_i)|^{2}),
\] where $\sigma_i$ in $\theta(\sigma_i)$ is a section with respect to the divisor $mD$ while $\sigma_i$ in $\theta\theta_F(\sigma_i)$ is a section with respect to the divisor $mD+mF$.

The local weight for $h$ is 
\[
\frac{1}{2m} \log (\sum_i|\theta_B(f_*\sigma_i)|^{2}).
\] Thus $f^*h$ has the local weight
\[
f^*\left(\frac{1}{2m}\log (\sum_i|\theta_B(f_*\sigma_i)|^{2})\right)
\] under the trivialization $f^*\theta_B$  (see \eqref{eq: pullback of metric}). The claim follows from
\[
f^*(\theta_B(f_*\sigma_i)) = (f^*\theta_B)(f^*(f_*\sigma_i))=(\theta \theta_F)(\sigma_i).
\]
\end{proof}

\section{Metrics for the abundant nef parts}\label{sec: 4}

\subsection{Approximations for the abundant nef part}\label{subsec: Algebraic approximation for the abundant nef part}

Let $\pi: X \to \De$ be a projective contraction to the disc $\De$. Suppose that $f: X' \to X/\De$ is a projective modification and $M'$ is a $\Qq$-divisor on $X'$. Moreover, suppose that $g: X' \to Z/\De$ is a proper contraction such that $M' \sim_\Qq g^*H'/\De$, where $H'$ is a nef and big $\Qq$-divisor over $\De$.

\[
\xymatrix{
X \ar[d]_\pi  & X' \ar[d]^g \ar[l]_f\\
\De  & Z\ar[l]_\tau }\\
\]  

Take projective resolutions $\ti f: X'' \to X', \ti\tau: Z'' \to Z$ such that $\ti g \coloneqq {\ti\tau}^{-1}\circ g\circ{\ti{f}}: X'' \to Z''$ is a morphism. By Hironaka's Chow lemma (see \cite[Chapter VII, Theorem 2.8]{Pet94b}), we can further assume that $Z'' \to \De$ and $\ti g$ are projective.  Let $M''=\ti f^*M'$, then $M''$ is the pull-back of the nef and big$/\De$ divisor $\ti\tau^*H'$. Hence, replacing $X', Z, M'$ by $X'', Z'', M''$, we can assume that $Z$ is smooth and $g, \tau$ are projective.

As the Picard group $\Pic(\De)$ is trivial, for $k=\Zz, \Qq$ and $k$-divisors $B, D$ on $X$, $B \sim_{k} D/\De$ is the same as $B \sim_{k} D$.

The following lemma is a direct consequence of the negativity lemma (see \cite[Lemma 3.39]{KM98}). Note that such result also holds for normal complex spaces (see \cite[Proof of Lemma 3.40]{KM98}).

\begin{lemma}\label{le: negativity}
Let $X' \xrightarrow{f} X \to S$ be a g-pair with a nef$/S$ divisor $M'$ on $X'$ and the nef part $M =f_*M'$. Suppose that $M$ is $\Qq$-Cartier, then 
\begin{equation}\label{eq: pullback nef part}
f^*M = M'+\Upxi,
\end{equation} where $\Upxi \geq 0$ is an $f$-exceptional divisor. 
\end{lemma}

\begin{lemma}\label{le: a decomposition}
Under the above notation and assumptions. For a g-pair $X' \xrightarrow{f} X \xrightarrow{\pi} \De$ with the abundant nef part. Let $X_0 \subset X$ be the fiber over $0\in\De$. Let $\ti X_0$ be the strict transform of $X_0$ on $X'$. There exist divisors $E, F$ on $X'$ that satisfy the following:
\begin{enumerate}
\item $E, F$ are effective $\Qq$-divisors without common components,
\item $E$ is  $f$-exceptional,
\item for each $k \in \Nn$, there exists an ample $\Qq$-Cartier divisor $H_k$ on $Z$ such that $$M' + \frac 1 k E \sim_{\Qq} g^*H_k + \frac 1 k F,$$

\item $\ti X_0$ is not a component of $E \cup F$.
\end{enumerate}
\end{lemma}

\begin{proof}
There is a nef and big$/\De$ divisor $H'$ on a smooth complex space $Z$ such that $M' \sim_{\Qq} g^*H'$. Then for any $k \in \Nn$, $H'\sim_{\Qq} H_k+\frac 1 k B$ where $H_k$ is ample and $B \geq 0$ (see \cite[Proposition 2.6.1 (3)]{KM98}). Let $c=\mult_{X_0} (f_*g^*B)$, then 
\begin{equation}\label{eq: subtract from g^*B}
g^*B - c(\pi \circ f)^*(0) =g^*B   - c (\ti X_0 + E_f)
\end{equation} does not contain $\ti X_0$ in its support, where $E_f$ is some $f$-exceptional divisor (here $0\in \De$ is viewed as a divisor and $(\pi\circ f)^*(0)$ is the pull-back of $0$). However, $g^*B - c(\pi \circ f)^*(0)$ may not be effective. 

We have
\[
g^*B - c(\pi \circ f)^*(0) = g^*B - c(\tau \circ g)^*(0) = g^*(B - c\tau^*(0)).
\] Write
\begin{equation}\label{eq: subtract from B}
B - c\tau^*(0) = \Theta - E^-,
\end{equation} where $\Theta, E^-$ are $\Qq$-Cartier effective divisors without common components. 
But $g^*\Theta, g^*E^-$ may have common components. Let 
\[
\Lambda \coloneqq g^*\Theta \wedge g^*E^-
\] be the effective $\Qq$-Cartier divisor such that $$\mult_P \Lambda=\min\{\mult_P g^*\Theta, \mult_P g^*E^-\}$$ for each prime divisor $P$. Then $g^*\Theta- \Lambda, g^*E^-- \Lambda$ are effective divisors without common components.

By \eqref{eq: subtract from B},
\[
g^*B - c(\pi \circ f)^*(0) = g^*\Theta- g^*E^- = (g^*\Theta- \Lambda)-(g^*E^-- \Lambda).
\] Moreover, the negative component (i.e. the component with negative coefficients) of $g^*B - c(\pi \circ f)^*(0)$ is $f$-exceptional by \eqref{eq: subtract from g^*B}. Therefore, $g^*E^-- \Lambda$ is $f$-exceptional. By the choice of $c$ (see \eqref{eq: subtract from g^*B}), 
\[
\ti X_0 \not\subset \Supp ((g^*\Theta- \Lambda)\cup(g^*E^-- \Lambda)).
\]

Let $E = g^*E^- - \Lambda$ and $F = g^*\Theta-\Lambda$, then
\[
M' +\frac 1 k E \sim_{\Qq} g^*H_k+\frac 1 k F.
\]
\end{proof}

\subsection{Construction metrics for the abundant nef part}

Under the above notation and assumptions, for the morphism $f: X' \to X$, let $\ti X_0$ be the strict transform of $X_0$. Replacing $f$ by a projective resolution, we can assume that $\ti X_0$ is smooth. Let $f_0: \ti X_0 \to X_0$ be the corresponding morphism. Set $E_0\coloneqq E|_{\ti X_0}, F_0\coloneqq F|_{\ti X_0}$ which are well-defined $\Qq$-Cartier divisors by Lemma \ref{le: a decomposition} (4), and set $\Upxi_0\coloneqq \Upxi|_{\ti X_0}$  (see Lemma \ref{le: negativity} for the definition of $\Upxi$).

\begin{lemma}\label{le: a metric on X_0}
After shrinking $\De$ around $0$, there is a metric $h_k'$ (non-canonically constructed) for the $\Qq$-Cartier divisor $M'+\frac 1 k E$ such that $h'_k$ is a non-negative metric defined by global sections as \eqref{eq: definition of metric}. Moreover,
\begin{enumerate}
\item $h'_k|_{\ti X_0} \not\equiv \infty$, 
\item $\vphi_k' \approx \frac 1 k \vphi_F$, where $\vphi_k', \vphi_F$ are the local weights of $h_k', \hbar_F$ respectively, and 
\item $\vphi_{k,0}' \approx \frac 1 k \vphi_{F_0}$, where $\vphi_{k,0}', \vphi_{F_0}$ are the local weights of $h'_k|_{\ti X_0}, \hbar_{F_0}$ respectively.
\end{enumerate}
\end{lemma}
\begin{proof} 
We have $M' + \frac 1 k E \sim_\Qq g^*H_k + \frac 1 k F$. Because $H_k$ is ample, after shrinking $\De$ around $0$, $|mH_k|$ induces an embedding $Z \hookrightarrow \Pp_\De^r$ for some $m\in \Nn$. Pulling back a Fubini-Study metric of $\Oo_{\Pp_\De^r}(1)$, we have a positive and continuous metric $h_{mH_k}$ for $mH_k$ which is defined by global sections. Let 
\[
h_{H_k} \coloneqq (h_{mH_k})^{\frac 1 m} \quad \text{and} \quad  \chi_k \coloneqq g^*h_{H_k}.
\] Then $\chi_k$ is defined by the global sections of $g^*H_k$, and it is a smooth metric with non-negative curvature current for the semi-ample divisor $g^*H_k$. Multiplying with the non-negative metric $\hbar_F$ for $F$ (see \eqref{eq: for non-effective div}), we get the desired metric 
\[
h_k' \coloneqq \chi_k \cdot \hbar_F^{\frac 1 k}
\] for $g^*H_k + \frac 1 k F$, and thus for $M' + \frac 1 k E$ (see Remark \ref{rmk: metric for linear equiv}). As $\ti X_0 \not\subset \Supp F$, $h'|_{\ti X_0} \not\equiv \infty$. This is (1).

(2) follows as the local weight $\vphi_{\chi_k}$ for $\chi_k$ is a continuous function. We can choose trivializations such that $\vphi'_{k} = \vphi_{\chi_k} + \frac 1 k \vphi_F$.

We have $h_k'|_{\ti X_0} = \chi_k|_{\ti X_0} \cdot (\hbar_F|_{\ti X_0})^{\frac 1 k}$ and $\hbar_F|_{\ti X_0} = \hbar_{F_0}$. As the local weight for $ \chi_k|_{\ti X_0}$ is still a continuous function, (3) holds for the same reason as (2). 
\end{proof}

By Lemma \ref{le: a decomposition}, $E$ is $f$-exceptional, and thus 
\[
M =f_*M'=f_*(M'+\frac 1 k E)\sim_{\Qq} f_*(g^*H_k + \frac 1 k F).
\] Because $g^*H_k$ and $\frac 1 k F$ admit metrics $\chi_k$ and $\hbar_{F}^{\frac 1 k}$ which are both defined by global sections, $h_k' = \chi_k \cdot \hbar_{F}^{\frac 1 k}$ is also defined by global sections by Proposition \ref{prop: product of metric defined by global sections}. Moreover, $h_k'$ is a metric for $M'+\frac 1 k E$ (see Remark \ref{rmk: metric for linear equiv}). Thus $h_k \coloneqq f_*h_k'$ is a metric for $M$ which is still defined by global sections (see Section \ref{subsection: pushforward metrics}). 
 
 \begin{lemma}\label{le: comparing h with h_k'}
Under the above notation and assumptions, let $h_k= f_*h_k'$ be the metric for $M=f_*(M'+\frac 1 k E)$. Then
\begin{enumerate}
\item $h_k$ has non-negative curvature current,
\item $h_k|_{X_0} \not\equiv \infty$,
\item $f^*\vphi_k+\frac 1 k \vphi_E \approx \frac 1 k \vphi_F+\vphi_{\Upxi}$, where $\vphi_k, \vphi_E, \vphi_F, \vphi_{\Upxi}$ are local weights for $h_k, \hbar_E, \hbar_F,\hbar_\Upxi$ respectively, and
\item $f_0^*\vphi_{k,0}+\frac 1 k \vphi_{E_0} \approx \frac 1 k \vphi_{F_0}+\vphi_{\Upxi_0}$, where $\vphi_{k,0}, \vphi_{E_0}, \vphi_{F_0}, \vphi_{\Upxi_0}$ are local weights for $h_{k}|_{X_0}, \hbar_{E_0}, \hbar_{F_0}, \hbar_{\Upxi_0}$ respectively.

\end{enumerate}
\end{lemma}
\begin{proof}

As the metric defined by global sections always has non-negative curvature current, we have (1). 

For (2), by Lemma \ref{le: a metric on X_0}, $h'_k|_{\ti X_0} \not\equiv \infty$. By 
\[
f^*M = f^*(f_*(M'+\frac 1 k E))=(M'+\frac 1 k E)+\Upxi - \frac 1 k E,
\] Lemma \ref{le: comparing pullback metrics} implies $f^*h_k = h_k' \cdot \hbar_{\Upxi-\frac 1 k E}$, which is the same as
\begin{equation}\label{eq: relation of metrics}
f^*h_k \cdot \hbar_E^{\frac 1 k} = h_k' \cdot \hbar_\Upxi.
\end{equation} As $E, \Upxi$ are $f$-exceptional divisors, $h_E|_{\ti X_0} \not\equiv \infty, h_\Upxi|_{\ti X_0} \not\equiv \infty$. Hence $f^*h_k|_{\ti X_0} \not\equiv \infty$. This implies $h_k|_{\ti X_0} \not\equiv \infty$ as $f_0$ is a proper modification.

(3) follows from \eqref{eq: relation of metrics} and Lemma \ref{le: a metric on X_0} (2).

For (4), restricting \eqref{eq: relation of metrics} to $\ti X_0$, we have
\[
f^*h_k|_{\ti X_0} \cdot \hbar_E^{\frac 1 k}|_{\ti X_0} = h_k'|_{\ti X_0} \cdot \hbar_\Upxi|_{\ti X_0}.
\] Moreover, $\hbar_E|_{\ti X_0} = \hbar_{E_0}, \hbar_\Upxi|_{\ti X_0} = \hbar_{\Upxi_0}$. By $f^*h_k|_{\ti X_0} = f_0^*(h_k|_{X_0})$, we have
\[
f_0^*(h_k|_{X_0}) \cdot \hbar_{E_0}^{\frac 1 k} = h_k'|_{\ti X_0} \cdot \hbar_{\Upxi_0}.
\] The result follows from Lemma \ref{le: a metric on X_0} (3).
\end{proof}

\begin{remark}
In application, it is possible to work with $h_{\min}$, the metric with minimal singularities for $M$ (which also makes sense on complex spaces under suitable modifications). Comparing with $h_k$, we have $\vphi_{\min} \preceq \vphi_k$ where $\vphi_{\min}$ is the local weight for $h_{\min}$. In the following, we adopt the more direct approach to work with $h_k$. 
\end{remark}

\section{Extension theorems and invariance of plurigenera}\label{sec: 5}

\subsection{Multiplier ideal sheaves}

Recall that for a complex manifold $X$ and a Cartier divisor $L$ with a metric $h$. If the local weight of $h$ is $\vphi$, then the multiplier ideal sheaf $\Ii(h)$ is the sheaf of germs of holomorphic functions $\alpha$ such that $|\alpha|^2e^{-2\vphi}$ is locally integrable. Therefore, 
\[
\begin{split}
H^0(U, \Ii(h) \otimes \Oo_X(L)) =\{&s \in H^0(U, \Oo_X(L)) \mid |\theta(s)|^2e^{-2\vphi} \in L^1_{\rm loc}(U), \\
&\text{where~}\theta \text{~is a trivialization of~}\Oo_X(L) \}.
\end{split}
\]

We have the following observation.

\begin{proposition}\label{claim: same integration}
Let $X$ be a complex manifold. Let $L$ be a Cartier divisor with a metric $h$ and $E$ be a Cartier divisor with a metric $\hbar_E$. If $s \in \mathscr M_X(X)$ is a global section of both $L$ and $L+E$ (i.e. $\di(s)+L \geq 0, \di(s)+L+E \geq 0$), then $s \in H^0(X, \Oo_X(L) \otimes \Ii(h))$ iff $s \in H^0(X, \Oo_X(L+E) \otimes \Ii(h\hbar_E))$.
\end{proposition}
\begin{proof}
Suppose that $\theta_L: \Oo(L)|_U \simeq \Oo_U$ is a trivialization of $\Oo(L)$ on $U \subset X$. Shrinking $U$, let $\sigma_E \in \mathscr M_X(U)$ be a local equation of $E$. There is a trivialization $\theta_E: \Oo(E)|_U \simeq \Oo_U$ by multiplying $\sigma_E$. Therefore, locally at a point $x \in U$,
\[
\int |\theta_L(s)|^2 e^{-2\vphi_{h}}~dV_X
=\int |\theta_L\theta_E(s)|^2 e^{-2\vphi_{h}}e^{-2\log|\sigma_E|}~dV_X,
\]where $\vphi_h$ is the local weight of $h$ corresponding to $\theta_L$.
\end{proof}

\subsection{Extension theorem for twisted pairs}

We need a modification of \cite[Theorem 3.1]{Tak07} (see Lemma \ref{le: extension}). In fact, we only work in the setting of \cite[Theorem 3.1]{Tak07} instead of generalizing this result (see Remark \ref{rmk: generalize Tak07}).

Let $\mu: W \to X$ be a log resolution and $\tau: W \to X \to \De$ be the corresponding morphism. Let $N$ be a Cartier divisor on $W$ and $\ti h_N$ be a metric for $N$. Suppose that $X_0$ is a normal complex subspace and $Y_0 \subset \tau^{-1}(0)$ is the strict transform of $X_0$. Taking a higher log resolution, we can assume that $Y_0$ is a smooth divisor. Viewing $X_0$ as an effective Cartier divisor, we have $\mu^*X_0 = Y_0 + \Theta$ with $\Theta \geq 0$ a $\mu$-exceptional divisor. In particular, $\Theta_0 \coloneqq\Theta|_{Y_0} \geq 0$. Restricting
\[
K_W+Y_0 + \Theta \sim K_W
\] to $Y_0$ and by the adjunction formula, we have
\begin{equation}\label{eq: K_W|Y bigger}
K_{Y_0}+\Theta_0 \sim K_W|_{Y_0}.
\end{equation} Therefore, a section of $\Oo(mK_W|_{Y_0})$ corresponds to a section of $\Oo(mK_{Y_0} + m\Theta_0)$, and its integrability with respect to the metric $\hbar_{m\Theta_0}$ makes sense (see Remark \ref{rmk: metric for linear equiv}). 

We explain the natural isomorphism
\[
\ti\nu: \Oo(K_{Y_0} + \Theta_0) \simeq \Oo(K_W|_{Y_0})
\] following \cite[Page 5-6]{Tak07}. The pull-back of the coordinate function $t$ on $\De$ is regarded as a holomorphic function on $W$ and it is still denoted by $t$. Suppose that $w_1, \ldots, w_{n-1}, y_0$ are local coordinates with $Y_0=\{y_0=0\}$. Then $t= \xi y_0$ with $\Theta= \{\xi=0\}$. For a sufficiently small open set $U \subset Y_0$, let 
\[
\tau \omega_{Y_0} \coloneqq \tau dw_1 \wedge \cdots \wedge dw_{n-1} \in H^0(U, \Oo(K_{Y_0}+\Theta_0)),
\] where $\tau \in H^0(U, \Oo(\Theta_0))$ is a section. Let $\ti\tau$ be a local meromorphic extension of $\tau$ near $U$ such that $\di(\ti\tau)+\Theta \geq 0$. Then $\ti\tau dw_1 \wedge \cdots \wedge dw_{n-1}$ is a local extension of $\tau\omega_{Y_0}$. By abuse of notation, $\omega_{Y_0}$ is still used to denote a local extension of $\omega_{Y_0}$ to $W$. In fact, in what follows, we always restrict the above forms to $Y_0$, hence the results are independent of the choice of extensions. 

Then 
\begin{equation}\label{eq:star}
\ti \nu:  \tau\omega_{Y_0} \mapsto (\ti\tau \omega_{Y_0} \wedge dt)|_{Y_0}.
\end{equation} Note that this is well-defined because $dt$ is independent of the choices of the open set $U$. In order to write $( \omega_{Y_0} \wedge dt)|_{Y_0}$ in terms of $(dw_1 \wedge \cdots \wedge dw_{n-1}\wedge dy_0)|_{Y_0}$, note that 
\[
dt= \xi dy_0 + y_0 d\xi
\] and $dw_1 \wedge \cdots \wedge dw_{n-1}\wedge d\xi=0$. Thus
\begin{equation}\label{eq: ti nu}
\ti \nu:  \tau \omega_{Y_0} \mapsto (\ti\tau\xi \omega_{Y_0}  \wedge dy_0)|_{Y_0}.
\end{equation} Because $\Theta=\{\xi=0\}$ and $\di(\ti\tau)+\Theta \geq 0$, $\ti\tau\xi$ is holomorphic. Hence $(\ti\tau\xi \omega_{Y_0}  \wedge dy_0)|_{Y_0} \in H^0(U, \Oo(K_W|_{Y_0}))$. From such local expressions, $\ti\nu$ is seen to be an isomorphism. Moreover, the above discussion also gives a natural isomorphism 
\begin{equation}\label{eq:double dagger}
\Oo(mK_{Y_0} +m \Theta_0) \xrightarrow{\sim} \Oo(mK_W|_{Y_0}).
\end{equation}

\begin{lemma}\label{le: extension}
Under the above notation and assumptions. If $\ti h_N|_{Y_0}$ is well-defined, then for each $m\in \Nn$, as long as $\ti h_N \hbar_{m\Theta}$ has non-negative curvature current, any section of
\[
H^0(Y_0, \Oo(mK_W|_{Y_0} + N|_{Y_0}) \otimes \Ii(\ti h_N|_{Y_0} \hbar_{m\Theta_0}))
\] extends over $W$.
\end{lemma}
\begin{proof}
The argument is identically as that for \cite[Theorem 1]{Pau07} (a simplified and twisted version of \cite{Siu02}). Therefore, we just sketch the argument.

We show that the naturally defined morphisms $p,\nu,q$ (defined below) give the following commutative diagram
\begin{equation}\label{eq: diag 1}
\xymatrix{
&\Oo_W(mK_W+N) \ar[rd]^q \ar[ld]_p&\\
\Oo_{Y_0}(mK_{Y_0}+m\Theta_0+N|_{Y_0}) \ar[rr]^{\nu}& & \Oo_{Y_0}(mK_W|_{Y_0}+N|_{Y_0}).
}
\end{equation} 

It suffices to define the morphism when $N=0$. First, $p$ is defined through
\[
\Oo_W(mK_W) \xrightarrow{\sim} \Oo_W(mK_W+mY_0+m\Theta_0) \to \Oo_{Y_0}(mK_{Y_0}+m\Theta_0).
\] Suppose that locally $\Oo_{Y_0}(K_{Y_0})|_U=\omega_{Y_0}\Oo_{U}$, then the above is given in local expressions by
\[
\alpha (\omega_{Y_0}\wedge dy_0)^{\otimes m} \mapsto \frac{\alpha}{t^m}  (\omega_{Y_0}\wedge dy_0)^{\otimes m} \mapsto \frac{\alpha_0 y_0^m}{t^m} \omega_{Y_0}^{\otimes m},
\] where $\alpha$ is a local section of $\Oo_W$ and $\alpha_0 = \alpha|_{Y_0}$.

The $\nu$ is given as \eqref{eq:double dagger} and the $q$ is the natural restriction map. To show $q=\nu\circ p$, set $\xi_0 \coloneqq \xi|_{Y_0}$. As $t=\xi y_0$, we have
\[
\frac{\alpha_0 y_0^m}{t^m} \omega_{Y_0}^{\otimes m} = \frac{\alpha_0}{\xi_0^m} \omega_{Y_0}^{\otimes m}.
\] Thus
\[
p: \alpha(\omega_{Y_0}\wedge dy_0)^{\otimes m} \mapsto \frac{\alpha_0}{\xi_0^m}\omega_{Y_0}^{\otimes m}.
\] By $dt = \xi dy_0 + y_0 d\xi$, we have
\begin{equation}\label{eq: relation between forms}
\omega_{Y_0} \wedge dy_0 = \frac{1}{\xi} \omega_{Y_0} \wedge dt.
\end{equation} Thus 
\[
q: \alpha (\omega_{Y_0}\wedge dy_0)^{\otimes m} \mapsto \alpha_0 (\frac{1}{\xi_0} (\omega_{Y_0} \wedge dt)|_{Y_0})^{\otimes m}.
\] By \eqref{eq:star}, $\nu: \Oo(mK_{Y_0} + m\Theta_0) \xrightarrow{\sim} \Oo(mK_W|_{Y_0})$ is given by
\[
\nu: \tau_0 \omega_{Y_0}^{\otimes m} \mapsto \ti\tau_0(\omega_{Y_0} \wedge dt)^{\otimes m}|_{Y_0},
\] where $\tau_0 \in H^0(U, \Oo(m\Theta_0))$ and $\ti\tau_0$ is a meromorphic extension of $\tau_0$. Take $\tau_0 = \frac{\alpha_0}{\xi_0^m}$, we have $q=\nu \circ p$.

We claim that any section of 
\[
H^0(Y_0, \Oo(mK_{Y_0} + (N|_{Y_0}+m\Theta_0)) \otimes \Ii(\ti h_N|_{Y_0} h_{m\Theta_0}))
\] extends over $W$ (i.e. such section has preimage in $H^0(W, \Oo(mK_W+N))$). To show this, the argument of \cite[Theorem 1]{Pau07} goes through with the version of the Ohsawa-Takegoshi extension theorem in \cite[Theorem 2.1]{Pau07} replaced by \cite[Lemma 3.6]{Tak07}. 

Now, as $\nu$ induces the isomorphism
\[
\begin{split}
& H^0(Y_0, \Oo(mK_{Y_0} + (N|_{Y_0}+m\Theta_0)) \otimes \Ii(\ti h_N|_{Y_0} \hbar_{m\Theta_0}))\\
 \xrightarrow{\sim} &H^0(Y_0, \Oo(mK_W|_{Y_0} + (N|_{Y_0})) \otimes \Ii(\ti h_N|_{Y_0} \hbar_{m\Theta_0})),
\end{split}
\] the commutativity of the diagram \eqref{eq: diag 1} gives the desired result. 
\end{proof}

\begin{remark}\label{rmk: generalize Tak07}
When $N=0$, Lemma \ref{le: extension} is weaker than \cite[Theorem 3.1]{Tak07} because 
\[
H^0(Y_0, \Oo(mK_W|_{Y_0} ) \otimes \Ii(h_{m\Theta_0})) \subset H^0(Y_0, \Oo(mK_W|_{Y_0} )).
\]

It is likely that under the assumption of non-negative curvature current of $\ti h_N$, any section of
\[
H^0(Y_0, \Oo(mK_{Y_0} + m\Theta_0+N|_{Y_0}) \otimes \Ii(\ti h_N|_{Y_0}))
\] extends over $W$.
\end{remark}

\subsection{The construction of $\Gg_m(h)$}\label{subsec: Multiplier ideal sheaves and sections}

Let $X$ be a complex space with canonical singularities. Note that the canonical divisor $K_X$ is a Weil divisor which may not be Cartier. There exists $\ell \in \Nn$ such that $\ell K_X$ is Cartier. Let  $\mu: W \to X$ be a resolution. Let $w_i, i=1 \ldots, n$ be local coordinates on $W$ and $z_i, i=1 \ldots, n$ be local coordinates on $X_{\rm reg}$. There is a multiple-valued meromorphic function $J(\mu)$ such that
\[
(d\mu^*z_1 \wedge \cdots \wedge d\mu^*z_n)^{\otimes \ell} = J(\mu)^\ell (dw_1 \wedge \cdots \wedge dw_n)^{\otimes \ell}, 
\] where $J(\mu)^\ell$ is a (single-valued) meromorphic function. In what follows, over possibly singular locus, we are only concerned with $|J(\mu)|$, hence the multi-valuedness of $J(\mu)$ will not cause any problem.

If we write $E = \mu^*K_X-K_W$ with $E$ the $\mu$-exceptional $\Qq$-Cartier divisor, then $J(\mu)$ is the local equation of $-E$ (up to multiply a nowhere vanishing function). As $X$ has canonical singularities, $E \leq 0$ and $J(\mu)$ is a multiple-valued holomorphic function.

Let $L$ be a Cartier divisor. As $X$ has canonical singularities, for any $m\in\Nn$ ($mK_X$ may not be Cartier), there is a natural pull-back
\[
\mu^*\Oo_X(mK_X) \to \Oo_W(mK_W)
\] which induces the natural map
\[
\mu^*\Oo_X(mK_X+L) \to \Oo_W(mK_W+\mu^*L).
\] To be precise, suppose that on an open set $U \subset X_{\reg}$, we have 
\[
s\in H^0(U, \Oo_X(mK_X+L))
\] such that
\[
s = \alpha (dz_1 \wedge \cdots \wedge dz_n)^{\otimes m}
\] with $\alpha \in H^0(U, \Oo(L))$. Let
\[
d\mu^*z_1 \wedge \cdots \wedge d\mu^*z_n = J(\mu) dw_1 \wedge \cdots \wedge dw_n.
\] As $U$ is smooth, $J(\mu)$ is a single-valued holomorphic function on $U$. Then 
\begin{equation}\label{eq: explicit expression}
\mu^*s=\mu^*\alpha \cdot J(\mu)^m (dw_1 \wedge \cdots \wedge w_n)^{\otimes m},
\end{equation} where $\mu^*\alpha \cdot J(\mu)^m \in H^0(\mu^{-1}(U), \Oo_W(\mu^*L))$. As $E \leq 0$,
\[
\mu^*\Oo_X(mK_X+L) \hookrightarrow \Oo_W(mK_W+\mu^*L),
\] and thus $\mu^*\alpha \cdot J(\mu)^m (dw_1 \wedge \cdots \wedge w_n)^{\otimes m}$ extends over $W$.

Let $X$ be a normal complex space and $L$ be a Cartier divisor with a metric $h$. Assume that $\vphi$ is a local weight of $h$ under some trivialization. The following type of ideal sheaves generalizes multiplier ideal sheaves.

\begin{definition}[Definition of $\Gg_m(h)$]\label{def: G_m}
Under the above notation and assumptions. Let $\Gg_m(h)$ be the sheaf of germs of holomorphic functions such that for any $x\in X$,
\[
\begin{split}
\Gg_m(h)_x = \{&\alpha\in \Oo_{X,x} \mid |\mu^*\alpha|^2|J(\mu)|^{2m}e^{-2\mu^*\vphi} \in L_{\rm loc}^1(\mu^{-1}(U)), \\
& \text{where~}\mu: W \to X \text{~is a resolution and~} U \text{~is a neighborhood of~} x \}.
\end{split}
\]
\end{definition}

\begin{remark}\label{rem: global resolution}
(1) Because of \eqref{eq: pullback relation}, the above definition is independent of the choice of trivializations. However, $\alpha$ may depend on $\mu$. (2) For technical reasons (see Lemma \ref{lem: one resolution}), we require $\mu$ to be a (global) resolution of $X$ instead of a neighborhood of $x$.
\end{remark}

When $X$ is smooth, we certainly have $\Gg_m(h) \supset \Ii(h)$. Furthermore, if $m=1$, then $\Gg_1(h) = \Ii(h)$ by the change-of-variables formula (see \eqref{eq: integrability of pullback}). The following example shows that the inclusion $\Gg_m(h)\supset\Ii(h)$ may be strict.

\begin{example}\label{eg: not the same as multiplier ideal sheaf}
Let $x, y$ be the coordinates of $\Cc^2$. Let $h=e^{-\vphi}$ with $\vphi \coloneqq \frac 1 2 \log(|x|^4+|y|^4)$ be the metric for the trivial divisor. Then $1 \not\in \Ii(h)$. We claim that $1\in \Gg_2(h)$. In fact, let $\mu: W \to \Cc^2$ be the blow-up of the origin. Then $W=\{(x,y)\times[u:v] \in \Cc^2\times \Pp^1 \mid xv=yu\}$. It is covered by two pieces $U_1=\{(x,y,z) \in \Cc^2 \mid xz=y)\}$ and $U_2=\{(x,y,w) \in \Cc^2 \mid x=yw)\}$. By the symmetry of $\vphi$, it suffices to consider the $L_{\rm loc}^1$-property on one piece, say $U_1$. Choose $x, z$ as local coordinates on $U_1$. By $dy=xdz+zdx$, we have
\[
xdx\wedge dz =d\mu^*x \wedge d\mu^*y.
\] Thus $J(\mu)=x$. Then locally on $U_1$,
\[
\begin{split}
&\int 1 \cdot |J(\mu)|^4e^{-2\mu^*\vphi}~dV_{W}
\\
=&\int \frac{|x|^4}{|x|^4+|y|^4}~dV_{W}=\int \frac{1}{1+|z|^4}~dV_{W}<\infty.
\end{split}
\]
\end{example}

It is natural to ask the following question: 

\begin{question}
Is $\Gg_m(h)$ a coherent sheaf?
\end{question}

To study the above question, it seems that the $L^2$-extension theorem also needs to be generalized in the bimeromorphic setting.

\subsection{Proof of Theorem \ref{thm: sing extension}, Theorem \ref{thm: AG sing trivial boundary extension} and Corollary \ref{cor: sing invariant of plurigenera for g-pair}}

Let $\pi: X \to \De$ be a projective contraction from a $\Qq$-Gorenstein complex space $X$ to the disc $\De$. Let $L$ be a Cartier divisor on $X$. Assume that $X_0\subset X$ is the fiber over $0\in\De$, and it is a normal complex subspace. As $X_0$ is a Cartier divisor, by adjunction formula,
\[
(K_X+X_0)|_{X_0} \sim_\Qq K_{X_0}.
\] Thus $X_0$ is also $\Qq$-Gorenstein, and $(mK_X+L)|_{X_0} \sim_\Qq mK_{X_0}+L|_{X_0}$ as $X_0$ is linearly equivalent to $0$ on $X$. Note that we do not assume that $mK_{X_0}$ is Cartier. We explain the meaning of extending sections from $X_0$ to $X$.

\begin{lemma}\label{le: meaning of extension}
Under the above notation and assumptions, there exists a natural map
\[
H^0(X, \Oo(mK_X+L)) \to H^0(X_0, \Oo(mK_{X_0} + L|_{X_0})).
\]
\end{lemma}
\begin{proof}
Let $V_0 \subset X_0$ be the smooth locus of $X_0$. As $X_0$ is a Cartier divisor, there exists a neighborhood $V\supset V_0$ such that $V\subset X$ is a smooth open variety. Let $\ti V_0 = X_0 \cap {V}$, then $\ti V_0 \supset V_0$.

Let $j: \ti V_0 \to V$ be the closed embedding. Because $V$ is smooth, we have
\[
\Oo_V(mK_X+L) \to  j_*\Oo_{\ti V_0}(mK_X+L),
\] and $\Oo_{\ti V_0}(mK_X+L) \simeq \Oo(mK_{\ti V_0}+L|_{\ti V_0})$ as $\ti V_0$ is linearly equivalent to $0$ on $V$. As $\codim_{X_0}(X_0\backslash V_0) \geq 2$ and $\ti V_0 \supset V_0$, 
\[
H^0(\ti V_0, \Oo(mK_{\ti V_0}+L|_{\ti V_0})) \simeq H^0(X_0, \Oo(mK_{X_0}+L|_{X_0})).
\] Therefore, there exist natural maps
\[
\begin{split}
&H^0(X, \Oo(mK_X+L)) \to H^0(V, \Oo_V(mK_X+L))\\
\to &H^0(\ti V_0, \Oo(mK_{\ti V_0}+L|_{\ti V_0}))\to H^0(X_0, \Oo(mK_{X_0} + L|_{X_0})).
\end{split}
\]
\end{proof}

Assume that $X$ has canonical singularities. Let $\mu: W \to X$ be a log resolution of $(X, X_0)$. Let $Y_0$ be the strict transform of $X_0$. Suppose that $X_0$ has canonical singularities. Let $\mu_0 \coloneqq \mu|_{Y_0}: Y_0 \to X_0$ and $E_0 = \mu_0^*K_{X_0} - K_{Y_0}$. The following is the key extension lemma.

\begin{lemma}\label{le: extension for a fixed resolution}
Under the above notation and assumptions. Let $L$ be a Cartier divisor on $X$ with a non-negative metric $h$. Suppose that $h|_{X_0}$ is well-defined. Then for each $m\in \Nn$, any section $s \in H^0(X_0, \Oo_{X_0}(mK_{X_0}+L|_{X_0}))$ such that 
\[
\mu_0^*s \in H^0\left(Y_0, \Oo(mK_{Y_0}+ \mu_0^*(L|_{X_0})) \otimes \Ii(\mu_0^*(h|_{X_0}))\right)
\] extends over $X$ (i.e. $s$ has a preimage in $H^0(X, \Oo(mK_X+L))$ under the natural map in Lemma \ref{le: meaning of extension}).
\end{lemma}

\begin{proof}
Let $L_0 = L|_{X_0}, h_0 = h|_{X_0}$. By assumption (see Definition \ref{def: non-negative curvature on variety}), $\mu^*h$ has non-negative curvature current. Let $\mu^*X_0=Y_0+\Theta$ and $\Theta_0 = \Theta|_{Y_0}$. Then $\mu^*h \hbar_{m\Theta}$ has non-negative curvature current by  $\Theta \geq 0$.

By Proposition \ref{claim: same integration} and \eqref{eq: K_W|Y bigger},
\[
\begin{split}
& H^0(Y_0, \Oo(mK_{Y_0} + \mu_0^*L_{0}) \otimes \Ii(\mu_0^*h_0 ))\\
\subset  & H^0(Y_0, \Oo(mK_{Y_0} +m\Theta_0 + \mu_0^*L_{0}) \otimes \Ii(\mu_0^*h_0 \hbar_{m\Theta_0}))\\
\simeq & H^0(Y_0, \Oo(mK_W|_{Y_0}  + \mu_0^*L_{0}) \otimes \Ii(\mu_0^*h_0 \hbar_{m\Theta_0})).
\end{split}
\] By Lemma \ref{le: extension}, $\mu_0^*s$ extends over $W$. That is, there exists $\ti \omega \in H^0(W, \Oo(mK_W+\mu^*L))$ such that $\ti \omega|_{Y_0} = \mu_0^*s$. 

We have the following diagram
\begin{equation}\label{eq: diag2}
\xymatrix{
H^0(W, \Oo(mK_W+\mu^*L))  \ar[r]^a & H^0(Y_0, \Oo(mK_W|_{Y_0}+\mu_0^*L_0)) \\
H^0(X, \Oo(mK_X+L))  \ar[u]_\simeq^b \ar[r]^d  & H^0(X_0, \Oo(mK_{X_0}+L_0)), \ar@{^{(}->}[u]_c }
\end{equation}  where $a$ is the restriction map, $d$ comes from Lemma \ref{le: meaning of extension}, $b$ comes from that $X$ has canonical singularities, and $c$ is an inclusion which comes from \eqref{eq: K_W|Y bigger} and $X_0$ has canonical singularities. The desired extension follows from the diagram chasing once we know that the diagram is commutative. 

To check the commutativity of the diagram, we follow the notation in the discussion before Lemma \ref{le: extension}.
Besides, it is enough to work locally in the smooth loci of $X_0, Y_0$ and $W$. Then
\[
b: \Oo_X(mK_X+L) \to \Oo_W(m\mu^*K_X+\mu^*L)\to \Oo(mK_W+\mu^*L)
\] is given by
\[
\beta (\omega_{X_0} \wedge dt)^{\otimes m} \mapsto \mu^*\beta (\mu^*\omega_{X_0} \wedge dt)^{\otimes m} \mapsto \mu^*\beta (J(\mu_0)\omega_{Y_0} \wedge dt)^{\otimes m},
\] where $\beta$ is a local section of $L$, $\omega_{X_0}$ and $\omega_{Y_0}$ are local generators of $\Oo(K_{X_0})$ and $\Oo(K_{Y_0})$ respectively, and $\mu^*\omega_{X_0} = J(\mu_0) \omega_{Y_0}$. Note that the $t$ in each term is the corresponding pull-back of the coordinate $t$ on $\De$. 

The map
\[
\begin{split}
c: \Oo(mK_{X_0}+L_0) \to \Oo(m\mu_0^*K_{X_0} +\mu_0^*L_0) &\hookrightarrow \Oo(mK_{Y_0}  +\mu_0^*L_0) \\
&\to \Oo(mK_W|_{Y_0}+\mu_0^*L_0)
\end{split}
\] is given by
\[
\beta_0\omega_{X_0}^{\otimes m} \mapsto \mu_0^*\beta_0 (\mu_0^*\omega_{X_0})^{\otimes m} \mapsto \mu_0^*\beta_0 (J(\mu_0)\omega_{Y_0})^{\otimes m} \mapsto  \mu_0^*\beta_0 J(\mu_0)^m (\omega_{Y_0}\wedge dt|_{Y_0})^{\otimes m},
\] where $\beta_0$ is a local section of $L_0$ and $\mu_0^*\omega_{X_0}=J(\mu_0) \omega_{Y_0}$. The last map is \eqref{eq:double dagger}.

For the same local section $\beta$ of $L$, the map $a$ is given by
\[
a: \mu^*\beta (\omega_{Y_0} \wedge dt)^{\otimes m}\mapsto 
(\mu^*\beta|_{Y_0})(\omega_{Y_0}\wedge dt|_{Y_0})^{\otimes m}.
\] (Strictly speaking, locally around $Y_0$, a section of $\Oo(mK_W+\mu^*L)$ should be represented by $\mu^*\beta (\omega_{Y_0} \wedge dy_0)^{\otimes m}$. Note $\omega_{Y_0} \wedge dy_0 = \frac{1}{\xi} \omega_{Y_0} \wedge dt$ by \eqref{eq: relation between forms}.)

Finally, the map $d$ is given by
\[
d: \beta (\omega_{X_0} \wedge dt)^{\otimes m} \mapsto \beta|_{Y_0} \omega_{X_0}^{\otimes m}.
\]

Thus the diagram \eqref{eq: diag2} commutes.
\end{proof}

\begin{remark}\label{rmk: above lemma is enough}
Lemma \ref{le: extension for a fixed resolution} is enough to show the desired extension result Theorem \ref{thm: AG sing trivial boundary extension}. The reason is that for any section of $H^0(X_0, \Oo_{X_0}(m(K_{X_0}+ M|_{X_0})))$, the integrability requirement is always satisfied when taking the resolution $W=X'$ (see \eqref{eq: integrable}, the key is that after taking a resolution, we have an extra weight $\vphi_{-mB_0}$).
\end{remark}

\begin{lemma}\label{le: comparing sections}
Let $X$ be a $\Qq$-Gorenstein complex space. Let $g: V \to X$ and $u: W  \to V$  be resolutions. Set $f=g\circ u: W \to X$. Let $L$ be a Cartier divisor on $X$ with a metric $h$. For any $m\in \Nn$ and $s \in H^0(X, \Oo(mK_X+L))$, if $g^*s \in H^0(V, \Oo(mK_V+g^*L) \otimes \Ii(g^*h))$, then $f^*s \in H^0(W, \Oo(mK_W+f^*L) \otimes \Ii(f^*h ))$. 
\end{lemma}
\begin{proof}
By $V$ smooth and $g^*s \in H^0(V, \Oo(mK_V+g^*L))$, we have $f^*s \in H^0(W, \Oo(mK_W+f^*L))$. Hence we only need to check the integrability. 

Let $v_i, i=1, \ldots, n$ and $w_i, i=1, \ldots, n$ be local coordinates of $V$ and $W$ respectively. Assume that $g^*s = \sigma (dv_1 \wedge \cdots \wedge dv_n)^{\otimes m}$ with $\sigma$ a local section of $g^*L$, then locally we have
\[
\int \|\sigma\|^2_{g^*h} ~dV_{V} < \infty.
\] By the change-of-variables formula, this is the same as
\begin{equation}\label{eq: integrability of pullback}
\int \|u^*\sigma\|^2_{u^*g^*h} |J(u)|^2 ~dV_{W} < \infty,
\end{equation} where $J(u)$ is the local equation of the $u$-exceptional divisor $K_{W}-u^*K_{V}$.

On the other hand, 
\[
f^*s = u^*(\sigma (dv_1 \wedge \cdots \wedge dv_n)^{\otimes m})=(u^*\sigma) J(u)^{2m} (dw_1 \wedge \cdots \wedge dw_n)^{\otimes m}.
\] Thus $f^*s \in H^0(W, \Oo(mK_W+f^*L) \otimes \Ii(f^*h ))$
means that locally 
\[
\int \|u^*\sigma\|^2_{u^*g^*h} |J(u)|^{2m} ~dV_{W} < \infty.
\] As $m \geq 1$ and $J(u)$ is holomorphic, the claim follows.
\end{proof}

\begin{lemma}\label{lem: one resolution}
Let $X$ be a compact complex space with canonical singularities. Let $L$ be a Cartier divisor on $X$ with a metric $h$. For some $m\in\Nn$, $s\in H^0(X, \Oo_X(mK_X+L)\otimes \Gg_m(h))$ if and only if there exists a resolution $\mu: W \to X$ such that
\[
\mu^*s \in H^0(W, \Oo_W(mK_W+\mu^*L)\otimes\Ii(\mu^*h)).
\]
\end{lemma}
\begin{proof}
The sufficient part follows from the definition. In the following, we show the necessary part.

Let $\vphi$ be the local weight for $h$. By $X$ compact, there are open sets $U_j, 1 \leq j \leq k$ and projective resolutions $\mu_j: W_j \to X, 1 \leq j \leq k$ such that $s\in H^0(U_j, \Oo_X(mK_X+L))$ and
\begin{equation}\label{eq: loc L1}
|\mu^*_j\alpha|^2|J(\mu_j)|^{2m}e^{-2\mu_j^*\vphi} \in L^1_{\rm loc}(\mu_j^{-1}(U_j)),
\end{equation}
where $s=\alpha\cdot(\omega_X)^{\otimes m}$ with $\omega_X$ a local generator for $\Oo(K_X)$ and $\alpha$ a local section of $L$ (see Remark \ref{rem: global resolution} (2)). More precisely, we have
\[
\mu_j^*s=(\mu_j^*\alpha)\cdot J(\mu_j)^m(\omega_{W_j})^{\otimes m},
\] where $(\mu_j^*\alpha)\cdot J(\mu_j)^m$ is first defined on the smooth locus, but it extends as a local section of $\mu_j^*L$ on $W_j$ as $X$ has canonical singularities. 

Hence, \eqref{eq: loc L1} is the same as
\begin{equation}\label{eq: higher multiplier ideal space}
\mu_j^*s \in H^0(\mu^{-1}_j(U_j), \Oo_{W_j}(mK_{W_j}+\mu_j^*L)\otimes \Ii(\mu_j^*h)).
\end{equation}

Let $\mu: W \to X$ be a resolution such that $\mu$ factors through each $\mu_j, 1 \leq j \leq k$. We claim that $\mu$ satisfies the desired property. By construction, the natural morphism $\nu_j: \mu^{-1}(U_j) \to \mu^{-1}_j(U_j)$ is a resolution. By \eqref{eq: higher multiplier ideal space} and Lemma \ref{le: comparing sections},
\[
\nu_j^*(\mu_j^*s) \in H^0(\mu^{-1}(U_j), \Oo_{W}(mK_W+\mu^*L)\otimes\Ii(\mu^*h)).
\] As $\nu_j^*(\mu_j^*s) = (\mu^*s)|_{\mu^{-1}(U_j) }$, the claim follows.
\end{proof}

\begin{lemma}\label{le: extend birational morphism}
Let $X$ be a complex space, and $X_0 \subset X$ be a reduced irreducible compact complex subspace of codimension $1$. Suppose that $\nu_0: \ti X_0 \to X_0$ is a proper modification. Then there exist smooth complex spaces $Y_0\subset Y$ and resolutions $\mu: Y \to X, \tau_0: Y_0 \to \ti X_0$ such that $\mu|_{Y_0}=\nu_0\circ\tau_0$.
\end{lemma}
\begin{proof}
By Hironaka's Chow lemma (see \cite[Chapter VII, Theorem 2.8, Corollary 2.9]{Pet94b}), there exist a blow-up $\sigma_0: Z_0 \to X_0$ of an ideal sheaf $\mathfrak I_0$ on $X_0$ and a proper modification $\mu_0: Z_0 \to \ti X_0$ such that $\sigma_0 = \nu_0\circ\mu_0$. Replacing $\nu_0$ by $\sigma_0$, we can assume that $\nu_0$ is the blow-up of the ideal sheaf $\mathfrak I_0$ on $X_0$. Let $\mathfrak I$ be the kernel of $\Oo_X \to \Oo_{X_0}/\mathfrak I_0$. Let $X'= \proj_X \oplus_{i=0}^\infty \mathfrak I^i \to X$ be the blow-up of $\mathfrak I$. Then there exist embeddings $\ti X_0\subset X_0 \times_{X} X' \subset X'$. Let $Y \to X'$ be a log resolution of $(X', \ti X_0)$ with $Y_0$ the strict transform of $\ti X_0$. Then the corresponding morphisms satisfy the claim.
\end{proof}

\begin{proof}[Proof of Theorem \ref{thm: sing extension}]
By Lemma \ref{lem: one resolution}, for a fixed $s\in H^0(X_0, \Oo_{X_0}(mK_{X_0}+L|_{X_0}) \otimes \Gg_m(h|_{X_0}))$, there is a resolution $\nu_0: \ti X_0 \to X_0$ such that $\nu_0^*s \in H^0(\ti X_0, \Oo_{\ti X_0}(mK_{\ti X_0}+\nu_0^*(L|_{X_0})) \otimes \Ii(\nu_0^*(h|_{X_0})))$. By Lemma \ref{le: extend birational morphism}, there are smooth complex spaces $Y_0 \subset Y$ and resolutions $\mu: Y \to X, \tau_0: Y_0 \to \ti X_0$ such that $\mu_0\coloneqq\mu|_{Y_0}=v_0\circ\tau_0$. By Lemma \ref{le: comparing sections}, 
\[
\mu_0^*s \in H^0(Y_0, \Oo_{Y_0}(mK_{Y_0}+\mu_0^*(L|_{X_0})) \otimes \Ii(\mu_0^*(h|_{X_0}))).
\] Then the claim follows from Lemma \ref{le: extension for a fixed resolution}.
\end{proof}

\begin{remark}
\cite[Theorem 1]{Tak07} does not assume that $X_0$ has canonical singularities. This is because the $m$-genus in \cite{Tak07} is defined by using its smooth model which does not coincide with our definition (take $L=0$) when the singularities are worse than the canonical singularities.
\end{remark}

Using this result, we show the extension theorem for g-pairs with abundant nef part. 

\begin{remark}\label{rmk: K_X not Cartier}
In Theorem \ref{thm: AG sing trivial boundary extension}, $mK_{X_0}$ is not assumed to be Cartier. Instead, it is just a Weil divisor on $X_0$. A priori, $K_{X_0}$ could be even not $\Qq$-Cartier. But if there exists $m$ such that $mM|_{X_0}$ is Cartier, then $K_{X_0}$ is $\Qq$-Cartier as $K_{X_0}+M_0$ is $\Qq$-Cartier (this is included in the definition of g-canonical singularities).
\end{remark}

\begin{proof}[Proof of Theorem \ref{thm: AG sing trivial boundary extension}]
As $\pi_*\Oo_X(mK_X+mM)$ is a coherent sheaf on $\De$, by Cartan's Theorem A, it suffices to show the extension after shrinking $\De$. Hence, by \cite[Main Theorem]{Kaw99}, we can assume that $X$ has canonical singularities.

Let $L=mM$ and $M'$ be the nef$/\De$ $\Qq$-divisor on $X'$ such that $f_*M'=M$ and satisfying Definition \ref{def: abundant gpair}. To apply Theorem \ref{thm: sing extension}, it suffices to construct a metric $h$ for $L$ satisfying conditions of Theorem \ref{thm: sing extension} and show
\[
H^0(X_0, \Oo_{X_0}(m(K_{X_0}+ M|_{X_0})))=H^0(X_0, \Oo_{X_0}(mK_{X_0}+L|_{X_0}) \otimes \Gg_m(h|_{X_0})).
\] 

By the adjunction formula (see \eqref{eq: adjunction on X_0}), we have a g-lc pair $(X_0, B_0+M_0)$ such that $K_{X_0}+B_0+M_0 = K_{X_0}+M|_{X_0}$. As $(X_0, B_0+M_0)$ has g-canonical singularities, $B_0=0$ and thus $M_0 = M|_{X_0}$. Then
\[
K_{\ti X_0}+D_0+M'_0 =f_0^*(K_{X_0}+M_0),
\] where $M_0'\coloneqq M'|_{\ti X_0}$ and $D_0 \leq 0$ is an $f_0$-exceptional divisor. Besides, 
\[
K_{\ti X_0}+B_0 =f_0^*K_{X_0},
\] where $B_0$ is an $f_0$-exceptional divisor. Because $(X_0, M_0)$ has g-canonical singularities and $K_{X_0}$ is $\Qq$-Cartier, $X_0$ also has canonical singularities, and thus $B_0 \leq 0$. By Lemma \ref{le: negativity}, $M'+\Upxi=f^*M$ with $\Upxi \geq 0$. We have 
\[
M_0'+\Upxi_{0}= f_0^*(M_0),
\] where $\Upxi_{0} \coloneqq \Upxi|_{X_0}$. Combining the above equations, we have 
\begin{equation}\label{eq: sing D =B+Upxi}
D_0=B_0+\Upxi_0.
\end{equation}

Shrinking $\De$ further, let $h_k$ be a metric for $M$ as in Lemma \ref{le: comparing h with h_k'} and set $h_{k,0}=h_k|_{X_0}$. For a fixed $m \in \Nn$, we claim that there exists $k \gg 1$ such that
\[
H^0(X_0, \Oo_{X_0}(m(K_{X_0}+ M|_{X_0}))) = H^0(X_0, \Oo_{X_0}(m(K_{X_0}+ M|_{X_0}))\otimes \Gg_m(h_{k,0}^{m})).
\] 
Then the theorem follows from Theorem \ref{thm: sing extension}.

For a section $s \in H^0(X_0, \Oo_{X_0}(m(K_{X_0}+ M|_{X_0})))$, assume that $\theta: \Oo(L)|_U \simeq \Oo_U$ is a trivialization and on $U_{\rm reg}$, $s$ can be written as $\alpha(dx_1 \wedge \cdots \wedge d x_{n-1})^{\otimes m}$. By Lemma \ref{lem: one resolution}, it suffices to show
\begin{equation}\label{eq: sing bounded}
\int \|f_0^*\alpha\|^2_{f_0^*(h^m_{k,0})} |J(f_0)|^{2m}~dV_{{\ti X_0}} < \infty.
\end{equation} 

\eqref{eq: sing bounded} holds for $k\gg 1$ by Lemma \ref{le: comparing h with h_k'} (4) and \eqref{eq: sing D =B+Upxi}. In fact, first note that $J(f_0)$ is the local equation of $-B_0$, hence $|J(f_0)|=e^{-\vphi_{B_0}}$, where $\vphi_{B_0}$ is the local weight for $\hbar_{B_0}$. Thus
\[
\|f_0^*\alpha\|^2_{f_0^*(h_{k,0}^{m})}|J(f_0)|^{2m}=\|f_0^*(\theta(\alpha))\|^2 \cdot e^{-2mf_0^*\vphi_{k,0}} \cdot e^{-2m\vphi_{B_0}},
\] where $\vphi_{k,0}$ is the local weight for $h_{k,0}$. By Lemma \ref{le: comparing h with h_k'} (4),
\[
f_0^*\vphi_{k,0}+\frac 1 k \vphi_{E_0} \approx \frac 1 k \vphi_{F_0}+\vphi_{\Upxi_0}.
\] By \eqref{eq: sing D =B+Upxi}, $mB_0 = mD_0 -m\Upxi_0$. Thus
\begin{equation}\label{eq: integrable}
\begin{split}
-mf_0^*\vphi_{k,0} - m\vphi_{B_0} &\approx m(-\frac 1 k \vphi_{F_0}-\vphi_{\Upxi_0}+\frac 1 k \vphi_{E_0})- m\vphi_{D_0} +m\vphi_{\Upxi_0} \\
& \approx m(-\frac 1 k \vphi_{F_0}+ \frac 1 k \vphi_{E_0} - \vphi_{D_0}). 
\end{split}
\end{equation} Note $-D_0 \geq 0$, and thus for a fixed $m$, we can take $k \gg 1$ such that the integrability of \eqref{eq: sing bounded} holds. In fact, it is enough to choose $k$ such that 
\[
\frac{2m}{k} \nu(\Theta_{\hbar_{F_0}}(F_0), y) <1 \text{~for all~} y \in \ti X_0,
\] where $\nu(\Theta_{h_{F_0}}(F_0), y)$ is the Lelong number of the curvature current at $y$.
\end{proof}

The following remark explains the crucial point of Theorem \ref{thm: AG sing trivial boundary extension}.

\begin{remark}\label{rmk: Paun's original thm is not enough}
Even in the smooth case, for the metric  $h=h_k$ constructed in the proof of Theorem \ref{thm: AG sing trivial boundary extension}, we may have
\[
H^0(X_0, \Oo_{X_0}(m(K_{X_0}+ M|_{X_0})))\subsetneqq H^0(X_0, \Oo_{X_0}(m(K_{X_0}+M|_{X_0})) \otimes \Ii(h|_{X_0})).
\] Hence \cite[Theorem 1]{Pau07} does not apply. In fact, under the notation of the proof of Theorem \ref{thm: AG sing trivial boundary extension}, assuming that $X_0$ is smooth, by the change-of-variables formula, we have
\[
\int \|\alpha\|_{h_0^m}^2 ~dV_{{X_0}}=
\int \|f_0^*\alpha\|_{f_0^*h_0^m}^2 |J(f_0)|^2~dV_{{\ti X_0}},
\] where $J(f_0)$ is the local equation of $K_{\ti X_0}-f_0^*K_{X_0}=-B_0\geq 0$ (c.f. \eqref{eq: sing bounded}). If $\ti \theta: \Oo(f_0^*(L|_{X_0}))|_{f_0^{-1}(U)} \simeq \Oo_{f_0^{-1}(U)}$ is a trivialization, then it becomes
\[
\int |\ti\theta(f_0^*\alpha)|^2e^{-2mf_0^*\vphi_{k,0}}e^{-2\vphi_{B_0}}~dV_{{\ti X_0}}.
\] However (c.f. \eqref{eq: integrable}),
\[
\begin{split}
-mf_0^*\vphi_{k,0}-\vphi_{B_0} &\approx m(-\frac 1 k \vphi_{F_0}-\vphi_{\Upxi_0}+\frac 1 k \vphi_{E_0})- \vphi_{D_0} +\vphi_{\Upxi_0}\\
& \approx m(-\frac 1 k \vphi_{F_0}+ \frac 1 k \vphi_{E_0} )- \vphi_{D_0}-(m-1)\vphi_{\Upxi_0}.
\end{split}
\] Note $\Upxi_0 \geq 0$, for $m\gg 1$, we do not have the integrability.

The new extension theorem works in this setting because the integrability requirement is for
\[
\int \|f_0^*\alpha\|^2_{f_0^*(h_{k,0}^{m})}|J(f_0)|^{2m} ~dV_{{\ti X_0}}.
\]  The extra $|J(f_0)|^{2m}$ makes the integral finite.
\end{remark}

\begin{proof}[Proof of Corollary  \ref{cor: sing invariant of plurigenera for g-pair}]
If $\Ff$ is a sheaf of $\Oo_\De$-module and $t\in \De$ is a  point, then $\Ff \otimes \Cc(t)=\Ff_t/m_t\Ff_t$, where $m_t\subset\Oo_{\De,t}$ is the maximal ideal corresponding to $t$ and $\Cc(t)\coloneqq \Oo_{\De,t}/m_t$. The following argument is similar to \cite[Proof of Theorem 1.1]{Tak07}.

Let $L=mM$ and $L_t=L|_{X_t}$. By the same argument as Lemma \ref{le: meaning of extension}, there is a natural map
\[
\pi_*\Oo_X(mK_{X}+L) \to H^0(X_t, \Oo_{X_t}(mK_{X_t}+L_t))
\] which is surjective by Theorem \ref{thm: AG sing trivial boundary extension}. For $\alpha \in m_t$ and $\sigma \in (\pi_*\Oo_X(mK_{X}+L))_t$, $(\alpha\otimes \sigma)|_{X_t}=0$, thus the above map induces the surjective map
\begin{equation}\label{eq: iso}
\pi_*\Oo_X(mK_{X}+L) \otimes \Cc(t) \to H^0(X_t, \Oo_{X_t}(mK_{X_t}+L_t)).
\end{equation} We show that this map is also injective. Let $U= X - {\rm Sing} X_t$. We claim that the following natural maps give a short exact sequence
\begin{equation}\label{eq: U SEC}
\begin{split}
0& \to \Oo_U(mK_X+(m-1)X_t+L) \to \Oo_U(mK_X+mX_t+L) \\
&\to\iota_* \Oo_{U \cap X_t}(mK_{X_t}+L_t) \to 0,
\end{split}
\end{equation} where $\iota: U \cap X_t \to U$. It is enough to check the exactness on stalks. Let $z \in U$. If $z \not\in X_t$, then $\iota_* \Oo_{U \cap X_t}(mK_{X_t}+L_t)_z=0$ and $\Oo_U(mK_X+(m-1)X_t+L)_z \simeq \Oo_U(mK_X+mX_t+L)_z$ as we can locally invert the defining equation of $X_t$. If $z \in U \cap X_t$, then by the choice of $U$, $z$ is a smooth point of $X_t$. As $X_t$ is Cartier, there is a smooth open set of $U$ which contains $z$. Then the exactness follows. Let $j: U \hookrightarrow X$. Pushing forward \eqref{eq: U SEC}, we have
\[
\begin{split}
0& \to \Oo_X(mK_X+(m-1)X_t+L) \to \Oo_X(mK_X+mX_t+L) \\
&\to j_*\iota_* \Oo_{U \cap X_t}(mK_{X_t}+L_t).
\end{split}
\] As $\codim_{X_t}(X_t - U\cap X_t) \geq 2$, the natural map
\[
\eta_*\Oo_{X_t}(mK_{X_t}+L_t) \to j_*\iota_* \Oo_{U \cap X_t}(mK_{X_t}+L_t)
\] is an isomorphism, where $\eta: X_t \to X$. In conclusion, there is an exact sequence
\[
0 \to \Oo_X(mK_X+(m-1)X_t+L) \rightarrow \Oo_X(mK_X+mX_t+L) \to \eta_*\Oo_{X_t}(mK_{X_t}+L_t).
\] Pushing forward by $\pi$ and taking the stalk at $t$, we have
\[
\begin{split}
0 &\to \pi_*\Oo_X(mK_X+(m-1)X_t+L)_t \rightarrow \pi_*\Oo_X(mK_X+mX_t+L)_t \\
&\to H^0(X_t, \Oo_{X_t}(mK_{X_t}+L_t)).
\end{split}
\] Let $\Ff \coloneqq \pi_*\Oo_X(mK_X+(m-1)X_t+L)$ and $\Gg\coloneqq\pi_*\Oo_X(mK_X+mX_t+L)$, then
\[
\begin{split}
&\Ker(\pi_*\Oo_X(mK_X+mX_t+L) \otimes \Cc(t) \to H^0(X_t, \Oo_{X_t}(mK_{X_t}+L_t)))\\
\simeq &(\Ff_t+m_t\Gg_t)/m\Gg_t.
\end{split}
\] We claim that $(\Ff_t+m_t\Gg_t)/m_t\Gg_t=0$. In fact, as $m_t=(z-t)\Oo_{\De,t}$, $\pi^*(z-t)$ is the defining equation of $X_t$, we have $\Ff_t\subset m_t\Gg_t$. Because $X_t \sim 0$, we have $\Oo_X(mK_X+(m-1)X_t + L) \simeq \Oo_X(mK_X + L)$. Thus \eqref{eq: iso} is an isomorphism.

Note that $\pi_*\Oo_X(mK_X+L)$ is a coherent sheaf. By the upper semi-continuity of 
\[
\dim_\Cc(\pi_*\Oo_X(mK_X+L)\otimes \Cc(t))
\] and the isomorphism \eqref{eq: iso}, $h^0(X_t, \Oo_{X_t}(mK_{X_t}+L_t))$ is upper semi-continuous. Fix  a $t_0 \in \De$, any section of $H^0(X_{t_0}, \Oo_{X_{t_0}}(mK_{X_{t_0}}+L_{t_0}))$ extends over $X$ by Theorem \ref{thm: AG sing trivial boundary extension}. Hence
\[
h^0(X_{t_0}, \Oo_{X_{t_0}}(mK_{X_{t_0}}+L_{t_0})) \leq h^0(X_t, \Oo_{X_t}(mK_{X_t}+L_t))
\] for a general $t \in \De$. Thus, $h^0(X_t, \Oo_{X_t}(mK_{X_t}+L_t))$ must be a constant for all $t\in \De$. 
\end{proof}

\subsection{Further discussions}

The following example shows that nefness of $M$ along does not guarantee the invariance of plurigenera. 

\begin{example}\label{eg: M nef}
Let $A/\Cc$ be an abelian variety and $A^\vee$ be its dual abelian variety. If $\Pic^0(A)$ is the identity component of the Picard variety of $A$, then $\Pic^0(A) = A^\vee$. Let $\mathcal{P}$ be the Poincar\'e bundle on $A \times A^\vee$. Suppose that $0\in A^\vee$ corresponds to $\mathcal P_0 \simeq \Oo_A$. Let $0\in \De \subset A^\vee$ be a disc containing $0$. Let $X = A \times \De$ and $\mathcal P_\De \coloneqq \mathcal P|_{X}$ be a line bundle on $X$. $\mathcal P_\De$ is nef over $\De$ as $\mathcal P_t$ is numerically trivial for each $t\in \De$. Moreover, $\Oo(K_{X}) = \Oo_{X}$. Note that for an abelian variety, a line bundle $\Ll \in \Pic^0(A)$ has global sections if and only if $\Ll\simeq \Oo_A$ (see \cite[Page 76, (vii)]{Mum70}). Therefore, for each $m\in \Nn$,
\[
h^0(X_{X_t}, \Oo_{X_t}(mK_{X_t})\otimes \mathcal P_\De^{\otimes m}|_{X_t})=\begin{cases}
1, & \text{ if~ } t=0,\\
0, & \text{ if~ } t \in \De-\{0\}.
\end{cases}
\] 
\end{example}

Next, recall that for a non-smooth family of varieties, we have
\begin{theorem}[{\cite[Theorem 1.1]{Tak07}}]
Let $\pi: X \to C$ be a proper surjective algebraic morphism with connected fibers from a complex variety $X$ to a smooth curve $C$. Assume that every fiber $X_t=\pi^{-1}(t)$ has only canonical singularities. Then $h^0(X_t, \Oo_{X_t}(mK_{X_t}))$ is independent of $t\in C$ for any positive integer $m$.
\end{theorem}

Such result has been established by \cite[Theorem 6]{Kaw99} under the additional assumption that each fiber is of general type. On the other hand, for klt singularities, local sections of fibers may not lift to global sections (see \cite[Example 4.3]{Kaw99b}). For g-pairs, the additional nef part introduces the singularities even if each $X_t$ is smooth. 

The above discussions show that both assumptions on the abundant nef parts and the g-canonical singularities are indispensable for Theorem \ref{thm: AG sing trivial boundary extension}
 and Corollary \ref{cor: sing invariant of plurigenera for g-pair}.

\bibliographystyle{alpha}
\bibliography{bibfile}
\end{document}